\documentclass[amssymb,twoside,11pt]{article}
\usepackage{latexsym,amsmath,graphicx}
\usepackage[dvipsnames]{xcolor}
\usepackage{placeins}
\usepackage{amsfonts}
\usepackage{enumitem}
\newtheorem{theorem}{Theorem}[section] 
\newtheorem{Lemma}[theorem]{{\bf Lemma}}

\newtheorem{ex}[theorem]{{\bf Example}}
\newtheorem{definition}{Definition}[section]
\numberwithin{equation}{section}
\newenvironment{proof}{\indent{\em Proof:}}{\quad \hfill
$\Box$\vspace*{2ex}}

\font\Bbb=msbm10 at 12pt

\newcommand{\R}{\mbox{\Bbb R}}

\begin{document}
\setcounter{page}{1}
\begin{center}
\vspace{0.4cm} {\large{\bf On the Boundary Value Problems of\\ $\Psi$-Hilfer Fractional Differential Equations}} \\
\vspace{0.5cm}
Ashwini D. Mali  $^{1}$\\
maliashwini144@gmail.com\\
\vspace{0.35cm}
Kishor D. Kucche $^{2}$ \\
kdkucche@gmail.com \\

\vspace{0.35cm}
$^{1,2}$ Department of Mathematics, Shivaji University, Kolhapur-416 004, Maharashtra, India.\\
\end{center}

\begin{abstract}
In the current paper, we derive the comparison results for the homogeneous and non-homogeneous linear initial value problem (IVP) for $\Psi$-Hilfer fractional differential equations. In the presence of upper and lower solutions, the obtained comparison results and the location of roots theorem utilized to prove the existence and uniqueness of the solution for the linear $\Psi$-Hilfer boundary value problem (BVP) through the linear non-homogeneous $\Psi$-Hilfer IVP. Assuming the existence of lower solution $w_0 $ and upper solution  $z_0 $, we establish the existence of minimal and maximal solutions for the nonlinear $\Psi$-Hilfer  BVP in the line segment $[w_0,\,z_0]$ of the weighted space  $C_{1-\,\gamma ;\, \Psi }\left( J,\,\R\right)$. Further, it demonstrated that the iterative Picard type sequences that began with lower and upper solutions respectively converges to a minimal and maximal solutions, and that started with any point on a line segment  converge to the exact solution of nonlinear $\Psi$-Hilfer BVP. Finally, an example is provided in support of the main results we acquired.
\end{abstract}
\def\baselinestretch{1.0}\small\normalsize
\allowdisplaybreaks
\section{Introduction}
The boundary value problems (BVPs) of fractional differential equations (FDEs) have significant applications \cite{Leung, Kaur, Constantinescu, Minhós} in mathematical physics,  financial mathematics, mathematical biology, biochemical system, biomedical engineering etc. Because of widespread applications, fractional-order BVPs were analyzed by numerous researchers for the existence of solutions by utilizing various procedures, for example, fixed point theorems  \cite{Su, Wangb, Arioua, Etemadb}, fixed point index theory \cite{Ma, Baip, Haob1, Haob2},  measure of non-compactness method  \cite{Goub1, Baitiche, Goub2}, upper and lower solution method \cite{Liub1, Zhaob, Jiab2, Long}, and so forth.

The technique of upper and lower solutions combined with the monotone iterative approach utilized effectively to derive sufficient conditions about the existence and uniqueness of solutions for the differential equations of integer order \cite{ Laddeb, Linb1, Linb2} and fractional order \cite{ Derbazi, Zhang, GWang, Akram, Shah, Linb3}	 subject to initial or boundary condition. The primary thought behind this method is to construct the monotonic sequences for the corresponding differential equations, where the initial approximations taken are upper and lower solutions.  At that point, it will be demonstrated that the sequences converge monotonically to the corresponding maximal and minimal solutions.

The fractional derivative with respect to another function is presented by  Kilbas et al. \cite{Kilbas} and Almeida \cite{Almeida} respectively in the sense of Riemann-Liouville and Caputo fractional derivative. Stimulated by the concept of  fractional derivative with respect to another function, in \cite{Vanterler1} Sousa and Oliveira  introduced  $\Psi$-Hilfer fractional derivative operator   $^H \mathcal{D}^{\alpha,\beta;\, \Psi}_{0+}(\cdot)$ and  investigated its important properties. Authors have proved that the $\Psi$-Hilfer derivative operator is more generalized and incorporates numerous fractional derivatives  as its special cases. For more details, we refer the reader to the paper \cite{Vanterler1}. For the analysis of nonlinear $\Psi$-Hilfer FDEs about existence and uniqueness of the solution, and its qualitative properties, such as Ulam-Hyers stability and various types of data dependence results, we refer the reader to the recent papers \cite{Luod, Abdoi, JVDC, Vanterler2,  Mali1, Mali2, Mali3, Mali4, Kharade}.
   
Motivated by the work mentioned above and the results obtained by  Lin et al. \cite{Linb3}, in the present paper, we discuss the existence and uniqueness results for linear and nonlinear BVPs of $\Psi$-Hilfer FDEs by the method of upper and lower solutions combined with monotone iterative technique. 

 To discuss the existence and uniqueness of solutions for BVPs of $\Psi$-Hilfer FDEs via  method of upper and lower solutions,  we need to derive the proper fractional differential
 inequalities in the setting of $\Psi$-Hilfer derivative as comparison results. Therefore, initially we obtain the comparison results for the homogeneous  linear initial value problem (IVP) for $\Psi$-Hilfer FDEs of the form
 \begin{align} \label{ln1}
 	\begin{cases}
    &^H \mathcal{D}^{\alpha,\,\beta\,;\, \Psi}_{0 +}y(t)- M y(t)
    = 0, ~t \in  (0,\,T], \\
  & \mathcal{I}_{0+}^{1-\,\gamma;\,\Psi } y(0)= 0,
  \end{cases}
    \end{align} 
and  non-homogeneous linear IVP for $\Psi$-Hilfer FDEs. Here $M>0~ (M \in \mathbb{R})$, $\Psi\in C^{1}([0,T],\mathbb{R})$ is an increasing function such that $\Psi'(t)\neq 0$, $ t\in [0,T]$,   $^H \mathcal{D}^{\alpha,\beta;\, \Psi}_{0+}(\cdot)$ is the $\Psi$-Hilfer fractional derivative of order $\alpha ~(0<\alpha<1)$ and type $\beta ~ (0\leq\beta\leq 1)$, $\gamma=\alpha+\beta\left( 1-\alpha\right)$ and $\mathcal{I}_{0+}^{1-\,\gamma;\,\Psi } (\cdot)$ is the $\Psi$-Riemann-Liouville fractional integral of order  $1-\,\gamma$.

With the help of acquired comparison results for the linear  $\Psi$-Hilfer IVP and the location of roots theorem \cite{Bartle}, assuming existence of lower and upper solution, we investigated the existence and uniqueness of the solution for  the following linear $\Psi$-Hilfer BVP
\begin{align}
   &^H \mathcal{D}^{\alpha,\,\beta\,;\, \Psi}_{0 +}y(t)- M y(t)
   = g(t),~t\in(0,\,T],  ~\label{eq513}\\
  & \mathcal{I}_{0+}^{1-\,\gamma;\,\Psi } y(0)= r\,\mathcal{I}_{0+}^{1-\,\gamma;\,\Psi } y(T),\label{eq514}
   \end{align}
 through the linear non-homogeneous  $\Psi$-Hilfer IVP,    where $g\in C_{1-\,\gamma ;\, \Psi }\left( J,\,\R\right),~J=[0,\,T]$ and $
0<r< \frac{1}{E_{\alpha, 1}\left(M (\Psi(T)-\Psi(0))^\alpha \right) }.
$
The two parameter Mittag-Leffler function $E_{\alpha, 1}(\cdot)$ and the weighted space $C_{1-\,\gamma ;\, \Psi }\left( J,\,\R\right)$ will be defined later in the preliminary section. 

Next, we consider the nonlinear $\Psi$-Hilfer BVP of the form 
 \begin{align}
   &^H \mathcal{D}^{\alpha,\,\beta\,;\, \Psi}_{0 +}y(t)- M y(t)
   = f(t, y(t)), ~t \in  (0,\,T],  ~\label{eq511}\\
  & \mathcal{I}_{0+}^{1-\,\gamma;\,\Psi } y(0)= r\,\mathcal{I}_{0+}^{1-\,\gamma;\,\Psi } y(T),\label{eq512}
   \end{align}
where  $f\left( \cdot, y(\cdot)\right) \in C_{1-\,\gamma ;\, \Psi }\left( J,\,\R\right)$ for each $y\in C_{1-\,\gamma ;\, \Psi }\left( J,\,\R\right)$.  Utilizing the existence and uniqueness results that are obtained for the linear non-homogeneous BVP of $\Psi$-Hilfer FDEs, and assuming that the nonlinear $\Psi$-Hilfer FDEs   \eqref{eq511}-\eqref{eq512} has lower solution $w_0 $ and upper solution  $z_0 $, it is proved that  there exist  minimal and maximal solutions on the line segment $[w_0,\,z_0]$ of the ordered Banach space (weighted space) $C_{1-\,\gamma ;\, \Psi }\left( J,\,\R\right)$.  Further, it is proved that the Picard iterative sequences $\{w_n\}_{n=1}^\infty$  and $\{z_n\}_{n=1}^\infty$  starting  respectively with  $w_0 $ (lower solution) and $z_0 $ (upper solution) are monotonic  in the ordered Banach space $C_{1-\,\gamma ;\, \Psi }\left( J,\,\R\right)$ and converges respectively to minimal and maximal solutions of the nonlinear BVP of $\Psi$-Hilfer  FDEs  \eqref{eq511}-\eqref{eq512}. 

Assuming that the function $f$ satisfies the one-sided Lipschitz condition, we have shown that the Picard iterative sequences beginning with any arbitrary point on the line segment  $[w_0,z_0]$ converges to the unique solution of the nonlinear $\Psi$-Hilfer  BVP  \eqref{eq511}-\eqref{eq512}. Further, the error bound between the  $n^{th} $ approximation $y_n$ and the exact solution $y^*$ of the nonlinear $\Psi$-Hilfer  BVP \eqref{eq511}-\eqref{eq512} is obtained with respect to the norm on the weighted space $C_{1-\,\gamma ;\Psi }\left( J,\,\R\right)$.  

Finally, an example is provided to illustrate the  existence and uniqueness results that we acquired through the method of lower and upper solution. 

The outcomes acquired in the present are the generalization of the results derived in \cite{Linb3} and can be achieved by taking  $\beta=1$ and $\Psi(t)=t$. The $\Psi$-Hilfer derivative $^H \mathcal{D}^{\alpha,\,\beta\,;\, \Psi}_{0 +} (\cdot)$ is generalized derivative operator that incorporates many notable fractional derivatives recorded in \cite{Vanterler1} as its special cases including most widely used derivative operators, such as, Riemann-Liouville derivative \cite{Kilbas}, Caputo derivative \cite{Kilbas},  Hadmard derivative\cite{Kilbas},  Erd$\acute{e}$ly-Kober derivative \cite{Kilbas}, Hilfer  derivative \cite{Hilfer}, Katugampola derivative \cite{Katugampola} etc. Along these lines, the outcomes acquired in the current paper are likewise valid for the fractional derivatives listed in \cite{Vanterler1} as  particular cases of the $\Psi$-Hilfer derivative.

The paper is composed as follows: In section 2, we collect a few definitions and essential outcomes about $\Psi$-Hilfer fractional derivative. Further, we have provided some results which assume a significant role in our analysis.   In Section 3, we prove the comparison results for $\Psi$-Hilfer  FDEs. Section 4 deals with the existence and uniqueness of the linear $\Psi$-Hilfer BVPs. In Section 5, the existence and uniqueness results are proved for the nonlinear $\Psi$-Hilfer BVPs \eqref{eq511}-\eqref{eq512}.  An example is provided in section 6 to verify the assurance of our primary outcomes. 

\section{Preliminaries} \label{preliminaries}

Let $[a,b]$ $(0<a<b<\infty)$ be a finite interval and $\Psi\in C^{1}([a,b],\mathbb{R})$ be an increasing function such that $\Psi'(t)\neq 0$, $ t\in [a,b]$. Consider the weighted space \cite{Vanterler1} 
\begin{equation*}
\mathcal{X}:= C_{1-\,\gamma ;\, \Psi }\left[ a,b\right] =\left\{ \mathfrak{h}:\left( a,b\right]
\rightarrow \mathbb{R}~\big|~\left( \Psi \left( t\right) -\Psi \left(
a\right) \right) ^{1-\,\gamma }\mathfrak{h}(t) \in C\left[ a,b\right]
\right\} ,\text{ }0< \gamma \leq 1,
\end{equation*}
endowed with the norm
\begin{equation}\label{space1}
\left\Vert \mathfrak{h}\right\Vert _{C_{1-\,\gamma ;\Psi }\left[ a,b\right] }=\underset{t\in \left[ a,b\right] 
}{\max }\left\vert \left( \Psi \left( t\right) -\Psi \left( a\right) \right)
^{1-\,\gamma }\mathfrak{h}(t) \right\vert,
\end{equation}
where $\gamma=\alpha+\beta(1-\alpha)$. Then, $\left(\mathcal{X}, \, \left\Vert \cdot\right\Vert _{C_{1-\,\gamma ;\Psi }\left( J,\,\R\right)}\right)$ is a partially ordered Banach space with the  partial order relation $\preceq$ defined by 
$$
x, \, y \in \mathcal{X},\, x\preceq y\, ~\text{if and only if}~ \, x(t) \leq y(t),\, t\in  (0,\,T].
$$

\begin{definition} [\cite{Kilbas}]
Let $\mu>0 ~(\mu \in \R)$, $\mathfrak{h}$ be an integrable function defined on $[a,b]$. Then, the $\Psi$-Riemann--Liouville fractional integral of a function $\mathfrak{h}$ of order $\mu$ with respect to $\Psi$ is given by 
\end{definition}
\begin{equation*}\label{P1}
I_{a+}^{\mu \, ;\Psi }\mathfrak{h}(t) =\frac{1}{\Gamma \left( \alpha
\right) }\int_{a}^{t}\Psi ^{\prime }\left( s\right) \left( \Psi \left(
t\right) -\Psi \left( s\right) \right) ^{\mu -1}\mathfrak{h}(s) ds.
\end{equation*}

\begin{definition}[\cite{Vanterler1}]
 Let $0<\alpha <1 $ and $\mathfrak{h} \in C^{1}([a,b],\mathbb{R})$. Then, the   $\Psi$-Hilfer fractional derivative   of a function $\mathfrak{h}$ of order $\alpha$ and type $\beta\left( 0\leq \beta \leq 1\right) $, is defined by
\begin{equation*}\label{HIL}
^{H}\mathcal{D}_{a+}^{\alpha ,\beta \, ;\Psi }\mathfrak{h}\left(t\right) =I_{a+}^{\beta \left(
1-\alpha \right) \, ;\Psi }\left( \frac{1}{\Psi ^{\prime }\left( t\right) }\frac{d}{dt}\right) I_{a+}^{\left( 1-\beta \right) \left( 1-\alpha
\right) \, ;\Psi }\mathfrak{h}(t).
\end{equation*}
\end{definition}

\begin{Lemma}[\cite{Kilbas,Vanterler1}]\label{lema2} 
Let $\mu_i>0~(i=1,2)$ and $\delta>0$. Then,
\begin{enumerate}
\item [(a)] $\mathcal{I}_{a+}^{\mu_1 \, ;\,\Psi }\mathcal{I}_{a+}^{\mu_2 \, ;\,\Psi }\mathfrak{h}(t)=\mathcal{I}_{a+}^{\mu_1+\mu_2  \, ;\,\Psi }\mathfrak{h}(t)$.
\item [(b)]$
\mathcal{I}_{a+}^{\mu \, ;\,\Psi }\left( \Psi \left( t\right) -\Psi \left( a\right)
\right) ^{\delta -1}=\frac{\Gamma \left( \delta \right) }{\Gamma \left(
\mu +\delta \right) }\left( \Psi \left(t\right) -\Psi \left( a\right)
\right) ^{\mu +\delta -1}.
$
\end{enumerate}
\end{Lemma}

\begin{Lemma}[\cite{Vanterler1}]\label{teo1} 
If $\mathfrak{h}\in C^{1}[a,b]$, $0<\alpha<1$ and $0\leq \beta \leq 1$, then
\begin{enumerate}[topsep=0pt,itemsep=-1ex,partopsep=1ex,parsep=1ex]
\item [(a)]
$
I_{a+}^{\alpha \, ;\Psi }\text{ }^{H}\mathcal{D}_{a+}^{\alpha ,\beta \, ;\Psi }\mathfrak{h}\left( t\right) =\mathfrak{h}\left( t\right) -\frac{\left( \Psi \left( t\right) -\Psi \left( a\right) \right) ^{\gamma -1}}{\Gamma \left( \gamma \right) }I_{a+}^{ 1-\gamma  \, ;\Psi }\mathfrak{h}(a)$.
\item [(a)]
 $
^{H}\mathcal{D}_{a+}^{\alpha ,\beta \, ;\Psi }I_{a+}^{\alpha \, ;\Psi }\mathfrak{h}\left( t\right)
=\mathfrak{h}\left( t\right).
$
\end{enumerate}
\end{Lemma}

\begin{Lemma}[\cite{Vanterler2}]\label{lema1}
 If $\mu >0$ and \, $0\leq \omega <1,$ then $\mathcal{I}_{a+}^{\mu \, ;\,\Psi }(\cdot)$ is bounded from $C_{\omega \, ;\,\Psi }\left[ a,b\right] $ to $C_{\omega \, ;\,\Psi }\left[ a,b\right] .$ In addition, if $\omega \leq \mu $, then $\mathcal{I}_{a+}^{\mu \, ;\,\Psi }(\cdot)$ is bounded from $C_{\omega \, ;\,\Psi }\left[ a,b\right] $ to $C\left[ a,b\right] $.
\end{Lemma}
 
 \begin{Lemma}[\cite{Mali1}]\label{lin1} 
 Let $g\in C_{1-\,\gamma ;\, \Psi }\left( J,\,\R\right)$ and $\eta\in\R$. Then, the solution of the Cauchy problem
 for FDEs with constant coefficient involving $\Psi$-Hilfer fractional derivative,
 	\begin{align*}
 	&  ^H \mathcal{D}^{\alpha,\,\beta\,;\, \Psi}_{0 +}y(t)- \eta y(t)
 	=g(t),~~t \in  (0,\,T],\\
 	& \mathcal{I}_{0+}^{1-\,\gamma;\,\Psi } y(0)= y_0,
 	\end{align*}
 	is given by 
 	\begin{align*}
 	y(t)&=y_0 (\Psi(t)-\Psi(0))^{\gamma-1}E_{\alpha,\,\gamma}\left(\eta (\Psi(t)-\Psi(0))^\alpha \right)\\
 	&+\int_{0}^{t} \Psi'(s)  (\Psi(t)-\Psi(s))^{\alpha-1}  E_{\alpha,\,\alpha}\left(\eta (\Psi(t)-\Psi(s))^\alpha \right) g(s) ds,~~t \in  (0,\,T].
 	\end{align*}
 \end{Lemma}

 \begin{Lemma}[\cite{Diethelm1}] \label{lem28}
Let $n_1, n_2 > 0, ~(n_i \in \mathbb{R}), ~i=1,2$.	Consider the two parameter Mittag--Leffler function defined by  $E_{n_1,n_2}(z)=\sum_{k=0}^{\infty}\frac{z^k}{\Gamma(n_1 k+n_2)},\, z\in\mathbb{C}$. Then, the power series defining $E_{n_1,n_2}(z)$  is convergent for all $z\in\mathbb{C}$.
\end{Lemma}

 \begin{Lemma}[Location of roots theorem (\cite{Bartle})] \label{lin3} 
Let $I=[a, b]$ and let $f: I\rightarrow \R$ be continuous on $I$. If $f(a)<0<f(b)$, or if $f(a)>0>f(b)$, then there exists a number $c\in (a, b)$ such that $f(c)=0$.
 \end{Lemma}

\begin{Lemma}[\cite{Linb1}] \label{lem29}
	Let $\mathcal{X}$ be a partially ordered Banach space, $\{x_n\}\subset \mathcal{X}$ a monotone sequence and relatively compact set, then $\{x_n\}$ is convergent.
\end{Lemma}

\begin{Lemma}[\cite{Linb1}] \label{lem30}
	Let $\mathcal{X}$ be a partially ordered Banach space, $x_n \preceq y_n(n=1, 2, 3,\cdots)$ if $x_n\rightarrow x^*,\, y_n\rightarrow y^*$ we have $x^*\preceq y^*$.
\end{Lemma}

   \begin{definition}
   Let $u \in C_{1-\,\gamma ;\, \Psi }\left( J,\,\R\right)$. We say that $u$ is a lower solution of the linear $\Psi$-Hilfer BVP \eqref{eq513}-\eqref{eq514} if 
   $$
 ^H \mathcal{D}^{\alpha,\,\beta\,;\, \Psi}_{0 +}u(t)- M u(t) \leq g(t) - a_{u}(t),~t\in(0,\,T],
   $$
   where 
   \begin{equation}\label{aa2.6}
   a_{u}(t)=
   \begin{cases}
 0, ~\text{if} \,\,  r\,\mathcal{I}_{0+}^{1-\,\gamma;\,\Psi } u(T)\geq  \mathcal{I}_{0+}^{1-\,\gamma;\,\Psi } u(0)\\
 \frac{1}{r}\left(  ^H \mathcal{D}^{\alpha,\,\beta\,;\, \Psi}_{0 +} \xi(t)- M\xi(t) \right) \left( \mathcal{I}_{0+}^{1-\,\gamma;\,\Psi } u(0)- r\,\mathcal{I}_{0+}^{1-\,\gamma;\,\Psi } u(T)\right),\\
 ~\text{if} \,\,r\,\mathcal{I}_{0+}^{1-\,\gamma;\,\Psi } u(T) <  \mathcal{I}_{0+}^{1-\,\gamma;\,\Psi } u(0),
   \end{cases}
     \end{equation}
and  \begin{equation}\label{aa2.6a}
   \xi(t)=\frac{\Gamma(2+\delta-\gamma)}{\Gamma(\delta+1)}   \frac{(\Psi(t)-\Psi(0))^{\delta}}{(\Psi(T)-\Psi(0))^{1-\,\gamma+\delta}},~\delta>0,\, t\in J.
   \end{equation}
   Clearly, $\xi\in C_{1-\,\gamma ;\, \Psi }\left( J,\,\R\right)$.
   \end{definition}

  \begin{definition}
   Let $v \in C_{1-\,\gamma ;\, \Psi }\left( J,\,\R\right)$. We say that $v$ is an upper solution of the linear $\Psi$-Hilfer BVP \eqref{eq513}-\eqref{eq514} if 
   $$
 ^H \mathcal{D}^{\alpha,\,\beta\,;\, \Psi}_{0 +}v(t)- M v(t) \geq g(t) + b_{v}(t),~t\in(0,\,T],
   $$
   where 
   \begin{equation}\label{aa2.7}
   b_{v}(t)=
   \begin{cases}
 0,  ~\text{if} \,\,    r\,\mathcal{I}_{0+}^{1-\,\gamma;\,\Psi } v(T)\leq  \mathcal{I}_{0+}^{1-\,\gamma;\,\Psi } v(0)\\
 \frac{1}{r}\left(  ^H \mathcal{D}^{\alpha,\,\beta\,;\, \Psi}_{0 +}\xi(t)- M\xi(t) \right) \left(r\,\mathcal{I}_{0+}^{1-\,\gamma;\,\Psi } v(T) - \mathcal{I}_{0+}^{1-\,\gamma;\,\Psi } v(0)  \right),\\
  ~\text{if} \,\, r\,\mathcal{I}_{0+}^{1-\,\gamma;\,\Psi } v(T) >  \mathcal{I}_{0+}^{1-\,\gamma;\,\Psi } v(0),
   \end{cases}
   \end{equation}
   where $\xi(t)$ is defined as in the equation \eqref{aa2.6a}.
   \end{definition}
   
\section{Comparison Theorems}
 \begin{theorem} \label{lem25}
Assume that $y \in C_{1-\,\gamma ;\, \Psi }\left( J,\,\R\right)$ and satisfies 
 \begin{align}
    &^H \mathcal{D}^{\alpha,\,\beta\,;\, \Psi}_{0 +}y(t)- M y(t)
    \leq 0, ~t \in  (0,\,T],  ~\label{eq515}\\
  & \mathcal{I}_{0+}^{1-\,\gamma;\,\Psi } y(0)\leq 0.\label{eq516}
    \end{align} 
 Then, $y(t)\leq0$, 
  $\,t\in(0,\,T]$.
 \end{theorem}
 \begin{proof}
 Consider the following linear IVP
 \begin{align}
     &^H \mathcal{D}^{\alpha,\,\beta\,;\, \Psi}_{0 +}y(t)- M y(t)
     =\sigma(t), ~t \in  (0,\,T],  ~\label{eq517}\\
    &\mathcal{I}_{0+}^{1-\,\gamma;\,\Psi } y(0)=y_0,\label{eq518}
     \end{align} 
     where $\sigma\in C_{1-\,\gamma ;\, \Psi }\left( J,\,\R\right)$.
 By Lemma \ref{lin1}, the linear IVP \eqref{eq517}-\eqref{eq518} has a unique solution  given by
 \begin{align}\label{eq519}
 y(t)&=y_0 (\Psi(t)-\Psi(0))^{\gamma-1}E_{\alpha,\,\gamma}\left(M (\Psi(t)-\Psi(0))^\alpha \right)\nonumber\\
 &+\int_{0}^{t} \Psi'(s)  (\Psi(t)-\Psi(s))^{\alpha-1}  E_{\alpha,\,\alpha}\left(M (\Psi(t)-\Psi(s))^\alpha \right) \sigma(s) ds,~t\in(0,\,T].
 \end{align}
 From  equations \eqref{eq516} and \eqref{eq518}, we have $y_0\leq0$. Since  $\Psi$ is an increasing function, we have $(\Psi(t)-\Psi(0))^{\gamma-1}\geq0,\,t\in J$ and $(\Psi(t)-\Psi(0))^{\alpha}\geq0,\,t\in J$.
 This gives, $E_{\alpha,\,\gamma}\left(M (\Psi(t)-\Psi(0))^\alpha \right)>0,\,t\in J$. Therefore, from \eqref{eq519}, we have
\begin{equation}\label{ka1}
y(t)\leq\int_{0}^{t} \Psi'(s)  (\Psi(t)-\Psi(s))^{\alpha-1}  E_{\alpha,\,\alpha}\left(M (\Psi(t)-\Psi(s))^\alpha \right) \sigma(s) ds,~t\in(0,\,T].
\end{equation}
 From equations \eqref{eq515} and \eqref{eq517},  it follows that $\sigma(t)\leq0, \, t\in J$.  Since  $\Psi: J\rightarrow \R$ is an increasing  continuous function, we have $\Psi'(t)>0,$ for $ \, t\in J$, $(\Psi(t)-\Psi(s))^{\alpha-1}\geq 0,$ for $  t\geq s \geq 0$ and
  $E_{\alpha,\,\gamma}\left(M (\Psi(t)-\Psi(s))^\alpha \right)>0,$ for $ t\geq s \geq 0$. Therefore,
  \begin{equation}\label{ka11}
\Psi'(s)  (\Psi(t)-\Psi(s))^{\alpha-1}  E_{\alpha,\,\alpha}\left(M (\Psi(t)-\Psi(s))^\alpha \right) \sigma(s)\leq 0, ~\text{for}\,~  0\leq s \leq t \leq T.
  \end{equation} 
  From \eqref{ka1} and \eqref{ka11}, we obtain 
  $$
  y(t)\leq0, ~t\in(0,\,T].
  $$
 \end{proof}
 \begin{theorem} \label{lem26}
 Let $y \in C_{1-\,\gamma ;\, \Psi }\left( J,\,\R\right)$  satisfies  
 \begin{align}
    &^H \mathcal{D}^{\alpha,\,\beta\,;\, \Psi}_{0 +}y(t)- M y(t)
    \leq -a_u(t), ~t \in  (0,\,T],  ~\label{eq520}\\
  & \mathcal{I}_{0+}^{1-\,\gamma;\,\Psi } y(0)\leq 0.\label{eq521}
    \end{align} 
 Then, $y(t)\leq0$, $\,t\in(0,\,T]$.
 \end{theorem}
 \begin{proof}
 In  the view of definition of $a_u$ given in \eqref{aa2.6}, we give the proof in following two cases.\\
\textbf{ Case 1:} If $ r\,\mathcal{I}_{0+}^{1-\,\gamma;\,\Psi } u(T)\geq  \mathcal{I}_{0+}^{1-\,\gamma;\,\Psi } u(0)$ then $a_u(t)=0,\,t\in(0,\,T]$. Thus, equations  \eqref{eq520}-\eqref{eq521} reduces to equations \eqref{eq515}-\eqref{eq516}. Applying Theorem \ref{lem25}, we obtain $y(t)\leq0$,  $t\in (0,\,T]$.\\
\textbf{ Case 2:} If $ r\,\mathcal{I}_{0+}^{1-\,\gamma;\,\Psi } u(T) <  \mathcal{I}_{0+}^{1-\,\gamma;\,\Psi } u(0)$ then
  $$
a_u(t)= \frac{1}{r}\left(  ^H \mathcal{D}^{\alpha,\,\beta\,;\, \Psi}_{0 +} \xi(t)- M\xi(t) \right) \left( \mathcal{I}_{0+}^{1-\,\gamma;\,\Psi } u(0)- r\,\mathcal{I}_{0+}^{1-\,\gamma;\,\Psi } u(T)\right),\,t\in(0,\,T].
  $$
Define
\begin{equation}\label{ka2}
\rho(t)=y(t)+\frac{1}{r} \xi(t) \left( \mathcal{I}_{0+}^{1-\,\gamma;\,\Psi } u(0)-  r\,\mathcal{I}_{0+}^{1-\,\gamma;\,\Psi } u(T) \right), ~t \in  (0,\,T].
\end{equation}
By assumption
$$
\mathcal{I}_{0+}^{1-\,\gamma;\,\Psi } u(0)- r\,\mathcal{I}_{0+}^{1-\,\gamma;\,\Psi } u(T)>0.
$$
Further, from  equation \eqref{aa2.6a}, we have $\xi(t)\geq0, \, t\in J$. Hence, we have
 \begin{equation}\label{ka3}
 \frac{1}{r} \xi(t) \left( \mathcal{I}_{0+}^{1-\,\gamma;\,\Psi } u(0)-  r\,\mathcal{I}_{0+}^{1-\,\gamma;\,\Psi } u(T) \right)\geq0,\,  t\in J.
\end{equation}
From \eqref{ka2} and \eqref{ka3}, it follows that
 \begin{equation}\label{eq551a}
y(t) \leq \rho(t), ~t \in  (0,\,T].
 \end{equation}
 Next, using \eqref{ka2}, we have
 \begin{align*}
 &^H \mathcal{D}^{\alpha,\,\beta\,;\, \Psi}_{0 +}\rho(t)- M \rho(t)\\
&=\, ^H \mathcal{D}^{\alpha,\,\beta\,;\, \Psi}_{0 +}\left[ y(t)+\frac{1}{r} \xi(t) \left( \mathcal{I}_{0+}^{1-\,\gamma;\,\Psi } u(0)-  r\,\mathcal{I}_{0+}^{1-\,\gamma;\,\Psi } u(T) \right)\right] \\
&\qquad-M\left[ y(t)+\frac{1}{r} \xi(t) \left( \mathcal{I}_{0+}^{1-\,\gamma;\,\Psi } u(0)-  r\,\mathcal{I}_{0+}^{1-\,\gamma;\,\Psi } u(T) \right)\right] \\
&=\, ^H \mathcal{D}^{\alpha,\,\beta\,;\, \Psi}_{0 +} y(t)- M y(t)+\frac{1}{r} \left[  ^H \mathcal{D}^{\alpha,\,\beta\,;\, \Psi}_{0 +}\xi(t)- M\xi(t)\right]  \left( \mathcal{I}_{0+}^{1-\,\gamma;\,\Psi } u(0)-  r\,\mathcal{I}_{0+}^{1-\,\gamma;\,\Psi } u(T) \right)\\
&=\,^H \mathcal{D}^{\alpha,\,\beta\,;\, \Psi}_{0 +} y(t)- M y(t)+ a_u(t), ~t \in  (0,\,T].
 \end{align*}
 Using the inequality \eqref{eq520}, above equation reduces to the following inequality
\begin{equation}\label{ka4}
 ^H \mathcal{D}^{\alpha,\,\beta\,;\, \Psi}_{0 +}\rho(t)- M \rho(t)\leq 0, ~t \in  (0,\,T].
 \end{equation}
Further, using \eqref{ka2} and Lemma \ref{lema2}(b), we have
 \begin{align}\label{ka5}
 & \mathcal{I}_{0+}^{1-\,\gamma;\,\Psi } \rho(t)\nonumber\\
  &= \mathcal{I}_{0+}^{1-\,\gamma;\,\Psi }y(t)+  \frac{1}{r}  \left( \mathcal{I}_{0+}^{1-\,\gamma;\,\Psi } u(0)-  r\,\mathcal{I}_{0+}^{1-\,\gamma;\,\Psi } u(T) \right)\mathcal{I}_{0+}^{1-\,\gamma;\,\Psi} \xi(t)\nonumber\\
 &=\mathcal{I}_{0+}^{1-\,\gamma;\,\Psi }y(t)\nonumber\\
 &+  \frac{1}{r}  \left( \mathcal{I}_{0+}^{1-\,\gamma;\,\Psi } u(0)-  r\,\mathcal{I}_{0+}^{1-\,\gamma;\,\Psi } u(T) \right)\mathcal{I}_{0+}^{1-\,\gamma;\,\Psi }\left[ \frac{\Gamma(2+\delta-\gamma)}{\Gamma(\delta+1)}\frac{(\Psi(t)-\Psi(0))^{\delta}}{(\Psi(T)-\Psi(0))^{1-\,\gamma+\delta}}\right]\nonumber\\
&=\mathcal{I}_{0+}^{1-\,\gamma;\,\Psi }y(t)+  \frac{1}{r}  \left( \mathcal{I}_{0+}^{1-\,\gamma;\,\Psi } u(0)-  r\,\mathcal{I}_{0+}^{1-\,\gamma;\,\Psi } u(T) \right)\frac{(\Psi(t)-\Psi(0))^{1-\,\gamma+\delta}}{(\Psi(T)-\Psi(0))^{1-\,\gamma+\delta}}.
 \end{align}
From  \eqref{eq521} and  \eqref{ka5}, it follows that
\begin{equation}\label{ka6}
   \mathcal{I}_{0+}^{1-\,\gamma;\,\Psi } \rho(0) \leq 0.
 \end{equation}
By applying Theorem \ref{lem25} to the inequalities \eqref{ka4} and \eqref{ka6}, we have
\begin{equation}\label{eq552a}
\rho(t)\leq0, ~t \in  (0,\,T].
\end{equation}
Further, from inequalities \eqref{eq551a} and \eqref{eq552a}, we obtain
 $y(t)\leq0, \, t\in  (0,\,T].$
  \end{proof}
  
  Following similar type of steps as in the proof of  Theorem \ref{lem26}, one can easily prove the following theorem.
   \begin{theorem} \label{lem27}
 Let $y \in C_{1-\,\gamma ;\, \Psi }\left( J,\,\R\right)$  satisfies 
   \begin{align}
      &^H \mathcal{D}^{\alpha,\,\beta\,;\, \Psi}_{0 +}y(t)- M y(t)
      \leq -b_v(t), ~t \in  (0,\,T],  ~\label{eq522}\\
     &\mathcal{I}_{0+}^{1-\,\gamma;\,\Psi } y(0)\leq 0.\label{eq524}
      \end{align} 
   Then, $y(t)\leq0, \, t\in  (0,\,T].$
   \end{theorem}
   

\section{Existence and uniqueness for the linear $\Psi$-Hilfer BVP}
In this section,  using the method of upper and lower solutions, we derive the existence and uniqueness results  for the linear $\Psi$-Hilfer BVP \eqref{eq513}-\eqref{eq514}. 
\begin{theorem}\label{lem531}
Assume that there exist upper and lower solutions 
$v, u \in C_{1-\,\gamma ;\, \Psi }\left( J,\,\R\right)$ respectively of the linear $\Psi$-Hilfer  BVP \eqref{eq513}-\eqref{eq514} such that $u\preceq v$. Then, the linear $\Psi$-Hilfer  BVP \eqref{eq513}-\eqref{eq514}  has a unique solution $y\in C_{1-\,\gamma ;\, \Psi }\left( J,\,\R\right)$ that satisfy $u\preceq y\preceq v$.
\end{theorem}
\begin{proof}
We give the  proof in  following two parts.\\
\textbf{Part I:} In this part we prove that the  linear $\Psi$-Hilfer BVP \eqref{eq513}-\eqref{eq514} has a unique solution. Consider the functions $p, q\in C_{1-\,\gamma ;\, \Psi }\left( J,\,\R\right)$ defined by
\begin{equation}\label{p}
p(t)=
\begin{cases}
r\, u(t),  ~\text{if} \,\,       r\,\mathcal{I}_{0+}^{1-\,\gamma;\,\Psi } u(T)\geq  \mathcal{I}_{0+}^{1-\,\gamma;\,\Psi } u(0)\\
r\, u(t)+\xi(t)\left( \mathcal{I}_{0+}^{1-\,\gamma;\,\Psi } u(0)- r\,\mathcal{I}_{0+}^{1-\,\gamma;\,\Psi } u(T)\right),   ~\text{if} \,\,     r\,\mathcal{I}_{0+}^{1-\,\gamma;\,\Psi } u(T) <  \mathcal{I}_{0+}^{1-\,\gamma;\,\Psi } u(0),
   \end{cases}
\end{equation}
and 
  \begin{equation}\label{q}
  q(t)=
   \begin{cases}
r\, v(t),   ~\text{if} \,\,     r\,\mathcal{I}_{0+}^{1-\,\gamma;\,\Psi } v(T)\leq  \mathcal{I}_{0+}^{1-\,\gamma;\,\Psi } v(0)\\
r\, v(t)-\xi(t) \left(r\,\mathcal{I}_{0+}^{1-\,\gamma;\,\Psi } v(T) - \mathcal{I}_{0+}^{1-\,\gamma;\,\Psi } v(0)  \right),    ~\text{if} \,\,     r\,\mathcal{I}_{0+}^{1-\,\gamma;\,\Psi } v(T) >  \mathcal{I}_{0+}^{1-\,\gamma;\,\Psi } v(0),
   \end{cases}
   \end{equation}
     where $\xi(t)$ is defined as in the equation \eqref{aa2.6a}.\\
Case 1:  If \, $ r\,\mathcal{I}_{0+}^{1-\,\gamma;\,\Psi } u(T)\geq  \mathcal{I}_{0+}^{1-\,\gamma;\,\Psi } u(0)$
then $p(t)=r\, u(t), \, t\in  (0,\,T]$. Therefore,
$$
\mathcal{I}_{0+}^{1-\,\gamma;\,\Psi } p(0)=r\, \mathcal{I}_{0+}^{1-\,\gamma;\,\Psi } u(0).
$$
Further,
$$
\mathcal{I}_{0+}^{1-\,\gamma;\,\Psi } p(T)=r\, \mathcal{I}_{0+}^{1-\,\gamma;\,\Psi } u(T)
\geq  \mathcal{I}_{0+}^{1-\,\gamma;\,\Psi } u(0)\\
=\frac{\mathcal{I}_{0+}^{1-\,\gamma;\,\Psi } p(0)}{r}.
$$
Thus, 
$$
r\,\mathcal{I}_{0+}^{1-\,\gamma;\,\Psi } p(T)\geq\mathcal{I}_{0+}^{1-\,\gamma;\,\Psi } p(0).
$$
Case 2: If \, $r\,\mathcal{I}_{0+}^{1-\,\gamma;\,\Psi } u(T) <  \mathcal{I}_{0+}^{1-\,\gamma;\,\Psi } u(0)$
then  $p(t)=r\, u(t)+\xi(t)\left( \mathcal{I}_{0+}^{1-\,\gamma;\,\Psi } u(0)- r\,\mathcal{I}_{0+}^{1-\,\gamma;\,\Psi } u(T)\right),$ $ \, t\in  (0,\,T]$. 
Therefore, 
\begin{equation} \label{aa2.6b}
\mathcal{I}_{0+}^{1-\,\gamma;\,\Psi } p(0)=r\,\mathcal{I}_{0+}^{1-\,\gamma;\,\Psi }  u(0)+\mathcal{I}_{0+}^{1-\,\gamma;\,\Psi } \xi(0)\left( \mathcal{I}_{0+}^{1-\,\gamma;\,\Psi } u(0)- r\,\mathcal{I}_{0+}^{1-\,\gamma;\,\Psi } u(T)\right).
\end{equation}
But from \eqref{aa2.6a} and Lemma \ref{lema2}(b), we have 
\begin{equation} \label{aa2.6c}
\mathcal{I}_{0+}^{1-\,\gamma;\,\Psi }\xi(t)=\frac{(\Psi(t)-\Psi(0))^{1-\,\gamma+\delta}}{(\Psi(T)-\Psi(0))^{1-\,\gamma+\delta}},\, t\in J.
\end{equation}
 From  \eqref{aa2.6b} and \eqref{aa2.6c}, it follows that 
$$
\mathcal{I}_{0+}^{1-\,\gamma;\,\Psi } p(0)=r\, \mathcal{I}_{0+}^{1-\,\gamma;\,\Psi } u(0).
$$ 
Further, 
\begin{equation} \label{aa2.6d}
\mathcal{I}_{0+}^{1-\,\gamma;\,\Psi } p(T)=r\, \mathcal{I}_{0+}^{1-\,\gamma;\,\Psi } u(T)+\mathcal{I}_{0+}^{1-\,\gamma;\,\Psi } \xi(T)\left( \mathcal{I}_{0+}^{1-\,\gamma;\,\Psi } u(0)- r\,\mathcal{I}_{0+}^{1-\,\gamma;\,\Psi } u(T)\right).
\end{equation}
Again from  \eqref{aa2.6c} and \eqref{aa2.6d}, it follows that
$$
\,\mathcal{I}_{0+}^{1-\,\gamma;\,\Psi } p(T)=\mathcal{I}_{0+}^{1-\,\gamma;\,\Psi } u(0)=\frac{\mathcal{I}_{0+}^{1-\gamma;\,\Psi } p(0)}{r}.
$$ 
Thus,
$$
r\,\mathcal{I}_{0+}^{1-\,\gamma;\,\Psi } p(T)= \mathcal{I}_{0+}^{1-\,\gamma;\,\Psi } p(0).
$$
From Case 1 and Case 2, we have
\begin{equation}\label{eq525}
\mathcal{I}_{0+}^{1-\,\gamma;\,\Psi } p(0)=r\, \mathcal{I}_{0+}^{1-\,\gamma;\,\Psi } u(0),
\end{equation}
\begin{equation}\label{eq526}
r\,\mathcal{I}_{0+}^{1-\,\gamma;\,\Psi } p(T)\geq\mathcal{I}_{0+}^{1-\,\gamma;\,\Psi } p(0).
\end{equation}
On the similar line for the  function $q$ one can obtain the following relations  
\begin{equation}\label{eq527}
\mathcal{I}_{0+}^{1-\,\gamma;\,\Psi } q(0)=r\, \mathcal{I}_{0+}^{1-\,\gamma;\,\Psi } v(0),
\end{equation}
\begin{equation}\label{eq528}
r\,\mathcal{I}_{0+}^{1-\,\gamma;\,\Psi } q(T)\leq\mathcal{I}_{0+}^{1-\,\gamma;\,\Psi } q(0).
\end{equation}
Next, our aim is to  prove that $p \in C_{1-\,\gamma ;\, \Psi }\left( J,\,\R\right)$ satisfies the following fractional differential inequality 
\begin{equation}\label{eq529}
^H \mathcal{D}^{\alpha,\,\beta\,;\, \Psi}_{0 +}p(t)- M p(t)\leq r\,g(t),~t\in(0,\,T].
\end{equation}

If  $ r\,\mathcal{I}_{0+}^{1-\,\gamma;\,\Psi } u(T)\geq  \mathcal{I}_{0+}^{1-\,\gamma;\,\Psi } u(0)$, then $a_u(t)=0,~t\in (0,\,T]$. Further,  $u$ is a lower solution of linear $\Psi$-Hilfer BVP   \eqref{eq513}-\eqref{eq514}. Therefore, we have
\begin{align*}
  ^H \mathcal{D}^{\alpha,\,\beta\,;\, \Psi}_{0 +}p(t)- M p(t)&= r \left[  ^H \mathcal{D}^{\alpha,\,\beta\,;\, \Psi}_{0 +} u(t)- M  u(t)\right] \\
        &\leq r\left[ g(t)-a_u(t) \right] \\
      &= r\,g(t),\, t\in  (0,\,T].
\end{align*}
If $r\,\mathcal{I}_{0+}^{1-\,\gamma;\,\Psi } u(T) <  \mathcal{I}_{0+}^{1-\,\gamma;\,\Psi } u(0), $ then
\begin{align*}
 & ^H \mathcal{D}^{\alpha,\,\beta\,;\, \Psi}_{0 +}p(t)- M p(t)\\
  &=\,^H \mathcal{D}^{\alpha,\,\beta\,;\, \Psi}_{0 +} \left[ r\, u(t)+\xi(t)\left( \mathcal{I}_{0+}^{1-\,\gamma;\,\Psi } u(0)- r\,\mathcal{I}_{0+}^{1-\,\gamma;\,\Psi } u(T)\right)\right] \\
  &\qquad- M  \left[ r\, u(t)+\xi(t)\left( \mathcal{I}_{0+}^{1-\,\gamma;\,\Psi } u(0)- r\,\mathcal{I}_{0+}^{1-\,\gamma;\,\Psi } u(T)\right)\right] \\
    &= r\left[  ^H \mathcal{D}^{\alpha,\,\beta\,;\, \Psi}_{0 +}  u(t)- M u(t)\right] \\
   & \qquad + \left[  ^H \mathcal{D}^{\alpha,\,\beta\,;\, \Psi}_{0 +} \xi(t)- M\xi(t)\right] \left( \mathcal{I}_{0+}^{1-\,\gamma;\,\Psi } u(0)- r\,\mathcal{I}_{0+}^{1-\,\gamma;\,\Psi } u(T)\right),\, t\in  (0,\,T].
\end{align*}
Using the definition of $a_u(t)$ and the fact that  $u$ is a lower solution of linear $\Psi$-Hilfer  BVP   \eqref{eq513}-\eqref{eq514} from above equation, we obtain 
 $$
 ^H \mathcal{D}^{\alpha,\,\beta\,;\, \Psi}_{0 +}p(t)- M p(t)\leq r\left[ g(t)-a_u(t) \right]+ r\,a_u(t)= r\,g(t),\, t\in  (0,\,T].
 $$
In both cases \, $ r\,\mathcal{I}_{0+}^{1-\,\gamma;\,\Psi } u(T)\geq  \mathcal{I}_{0+}^{1-\,\gamma;\,\Psi } u(0)$ and \,$r\,\mathcal{I}_{0+}^{1-\,\gamma;\,\Psi } u(T) <  \mathcal{I}_{0+}^{1-\,\gamma;\,\Psi } u(0), $ we have proved that the function $p\in C_{1-\,\gamma ;\, \Psi }\left( J,\,\R\right)$ defined in equation \eqref{p} satisfies the fractional differential inequality  \eqref{eq529}. On the similar line, one can prove that the function $q\in C_{1-\,\gamma ;\, \Psi }\left( J,\,\R\right)$ defined in equation \eqref{q} satisfies the following fractional differential inequality 
\begin{equation}\label{eq530}
^H \mathcal{D}^{\alpha,\,\beta\,;\, \Psi}_{0 +}q(t)- M q(t)\geq r\,g(t),~t\in(0,\,T].
\end{equation}
Define $\sigma(t)=p(t)-q(t), \,t\in (0,\,T].$  Then, $\sigma\in C_{1-\,\gamma ;\, \Psi }\left( J,\,\R\right)$.  By using fractional differential inequalities \eqref{eq529} and \eqref{eq530}, we obtain
\begin{align*}
^H \mathcal{D}^{\alpha,\,\beta\,;\, \Psi}_{0 +} \sigma(t)&=\, ^H \mathcal{D}^{\alpha,\,\beta\,;\, \Psi}_{0 +} p(t)-\, ^H \mathcal{D}^{\alpha,\,\beta\,;\, \Psi}_{0 +}q(t)\\
&\leq \left[r\,g(t) +M\,p(t)\right] -\left[ r\,g(t) +M \,q(t)\right] \\
&=M\, \sigma(t),\, t\in(0,\,T].
\end{align*}
Further, by using equations \eqref{eq525} and \eqref{eq527} and hypothesis, we obtain
\begin{align*}
\mathcal{I}_{0+}^{1-\,\gamma;\,\Psi } \sigma(0)
&=\mathcal{I}_{0+}^{1-\,\gamma;\,\Psi } p(0)-\mathcal{I}_{0+}^{1-\,\gamma;\,\Psi } q(0)\\
&=r\left[ \mathcal{I}_{0+}^{1-\,\gamma;\,\Psi } u(0)-\mathcal{I}_{0+}^{1-\,\gamma;\,\Psi } v(0)\right] \\
&\leq 0.
\end{align*}
We have proved that $\sigma\in C_{1-\,\gamma ;\, \Psi }\left( J,\,\R\right)$ satisfies
\begin{align*}
\begin{cases}
&^H \mathcal{D}^{\alpha,\,\beta\,;\, \Psi}_{0 +} \sigma(t)-M\, \sigma(t)\leq 0,\, t\in(0,\,T],\\
&\mathcal{I}_{0+}^{1-\,\gamma;\,\Psi } \sigma(0)\leq 0.
\end{cases}
\end{align*}
By applying Theorem \ref{lem25}, we obtain $\sigma(t)\leq 0,\, t\in (0,\,T].$ This gives 
$$
p(t)\leq q(t),\, t\in (0,\,T].
$$
Next, for any $\lambda\in\R$,  consider the following linear   $\Psi$-Hilfer FDEs subject to initial condition
\begin{align}\label{eq531}
\begin{cases}
&^H \mathcal{D}^{\alpha,\,\beta\,;\, \Psi}_{0 +} y(t)-M y(t)=g(t),\, t\in  (0,\,T],\\
&\mathcal{I}_{0+}^{1-\,\gamma;\,\Psi } y(t)|_{t=0}= \lambda.
\end{cases}
\end{align}
Then, by  Lemma \ref{lin1}, it has a unique  solution in $C_{1-\,\gamma ;\, \Psi }\left( J,\,\R\right)$ given by
\begin{align}\label{eq531a}
y(t, \lambda)&=\lambda\, (\Psi(t)-\Psi(0))^{\gamma-1}E_{\alpha,\,\gamma}\left(M (\Psi(t)-\Psi(0))^\alpha \right)\nonumber\\
&+\int_{0}^{t} \Psi'(s)  (\Psi(t)-\Psi(s))^{\alpha-1}  E_{\alpha,\,\alpha}\left(M (\Psi(t)-\Psi(s))^\alpha \right) g(s) ds, \, t\in  (0,\,T].
\end{align}
Since  $g \in C_{1-\,\gamma ;\, \Psi }\left( J,\,\R\right)$, the function $(\Psi(\cdot)-\Psi(0))^{1-\,\gamma}y(\cdot, \lambda)$ is continuous  on $J$ for each $\lambda\in \R$.\\
Define $\chi(t)=p(t)-r\, y(t, \lambda), \, t\in  (0,\,T],$ where $y(t, \lambda)$ is a solution of \eqref{eq531}. Then, $\chi\in C_{1-\,\gamma ;\, \Psi }\left( J,\,\R\right)$. Take any $\lambda\in \R$ such that 
\begin{equation}\label{am55}
\mathcal{I}_{0+}^{1-\,\gamma;\,\Psi } p(T)\leq \lambda\leq\mathcal{I}_{0+}^{1-\,\gamma;\,\Psi } q(T).
\end{equation}
  From equations \eqref{eq529} and \eqref{eq531}, we have
\begin{align*}
^H \mathcal{D}^{\alpha,\,\beta\,;\, \Psi}_{0 +} \chi(t)-M\, \chi(t)&=\,^H \mathcal{D}^{\alpha,\,\beta\,;\, \Psi}_{0 +} \left[ p(t)-r\, y(t, \lambda)\right] -M\,\left[ p(t)-r\, y(t, \lambda)\right]\\
&=\left[ ^H \mathcal{D}^{\alpha,\,\beta\,;\, \Psi}_{0 +}  p(t) - M\,p(t)\right] -r \left[ ^H \mathcal{D}^{\alpha,\,\beta\,;\, \Psi}_{0 +} y(t, \lambda) - M\,y(t, \lambda)\right] \\
&\leq r\,g(t)-  r\,g(t)=0, \, t\in  (0,\,T].
\end{align*}
Further, by using inequalities \eqref{eq526} and \eqref{am55} and initial condition in  \eqref{eq531}, we obtain
\begin{align*}
\mathcal{I}_{0+}^{1-\,\gamma;\,\Psi } \chi(0)
&=\mathcal{I}_{0+}^{1-\,\gamma;\,\Psi } p(0)-r \mathcal{I}_{0+}^{1-\,\gamma;\,\Psi } y(t, \lambda)|_{t=0}\\
&\leq r \left[\mathcal{I}_{0+}^{1-\,\gamma;\,\Psi }p(T)- \lambda\right]\\ 
&\leq 0.
\end{align*}
Therefore, $\chi\in C_{1-\,\gamma ;\, \Psi }\left( J,\,\R\right)$ satisfies
\begin{align*}
\begin{cases}
^H \mathcal{D}^{\alpha,\,\beta\,;\, \Psi}_{0 +} \chi(t)-M\, \chi(t)\leq 0,\, t\in(0,\,T],\\
\mathcal{I}_{0+}^{1-\,\gamma;\,\Psi } \chi(0)\leq 0.
\end{cases}
\end{align*}
By  applying Theorem \ref{lem25}, we obtain $\chi(t)\leq 0,\, t\in (0,\,T].$ This implies $p(t)\leq r\, y(t, \lambda),\, t\in (0,\,T]$. On the similar line one can prove that $ r\, y(t, \lambda)\leq q(t),\, t\in (0,\,T]$.  Therefore,
$$
p(t)\leq r\, y(t, \lambda)\leq q(t),\, t\in (0,\,T].
$$
Since $\Psi$-Riemann-Liouville fractional integral operator $\mathcal{I}_{0+}^{1-\,\gamma;\,\Psi }$ is monotonic, from above inequalities,  we obtain 
$$
\mathcal{I}_{0+}^{1-\,\gamma;\,\Psi }p(t)\leq r\,\mathcal{I}_{0+}^{1-\,\gamma;\,\Psi } y(t, \lambda)\leq \mathcal{I}_{0+}^{1-\,\gamma;\,\Psi }q(t),\, t\in (0,\,T].
$$ 
Therefore, we can write
\begin{equation}\label{eq522b}
\mathcal{I}_{0+}^{1-\,\gamma;\,\Psi }p(T)\leq r\,\mathcal{I}_{0+}^{1-\,\gamma;\,\Psi } y(T, \lambda)\leq \mathcal{I}_{0+}^{1-\,\gamma;\,\Psi }q(T),
\end{equation}
for any $\lambda\in\left[\mathcal{I}_{0+}^{1-\,\gamma;\,\Psi }p(T),\,\mathcal{I}_{0+}^{1-\,\gamma;\,\Psi }q(T) \right].$

Define
\begin{equation}\label{eq522a}
g(\lambda)= r\,\mathcal{I}_{0+}^{1-\,\gamma;\,\Psi } y(T, \lambda)-\lambda,\,\,\lambda\in\left[\mathcal{I}_{0+}^{1-\,\gamma;\,\Psi }p(T),\,\mathcal{I}_{0+}^{1-\,\gamma;\,\Psi }q(T) \right].
\end{equation}
Since two parameter Mittag-Leffler function is analytic, using it's series representation the equation \eqref{eq531a} can be written as
\begin{align*}
y(t, \lambda)
&=\lambda\, (\Psi(t)-\Psi(0))^{\gamma-1}\sum_{k=0}^{\infty}\frac{\left(M (\Psi(t)-\Psi(0))^\alpha \right)^k}{\Gamma(k\alpha+\gamma)}\\
&\qquad+\int_{0}^{t} \Psi'(s)  (\Psi(t)-\Psi(s))^{\alpha-1}  \sum_{k=0}^{\infty}\frac{\left(M (\Psi(t)-\Psi(s))^\alpha \right)^k}{\Gamma(k\alpha+\alpha)}g(s) ds\\
&=\lambda\sum_{k=0}^{\infty}\frac{M^k}{\Gamma(k\alpha+\gamma)}\left( \Psi(t)-\Psi(0)\right)^{k\alpha+\gamma-1}\\
&\qquad+\sum_{k=0}^{\infty}\frac{M^k}{\Gamma(\alpha(k+1))}\int_{0}^{t} \Psi'(s)  (\Psi(t)-\Psi(s))^{\alpha(k+1)-1}  g(s) ds\\
&=\lambda\sum_{k=0}^{\infty}\frac{M^k}{\Gamma(k\alpha+\gamma)}\left( \Psi(t)-\Psi(0)\right)^{k\alpha+\gamma-1}+\sum_{k=0}^{\infty}M^k\,\mathcal{I}_{0+}^{\alpha(k+1); \Psi } g(t),\, t\in (0,\,T].
\end{align*}
Using Lemma \ref{lema2}, from above equation, we obtain
\begin{align*}
&\mathcal{I}_{0+}^{1-\,\gamma;\,\Psi }y(t, \lambda)\\
&=\lambda\sum_{k=0}^{\infty}\frac{M^k}{\Gamma(k\alpha+\gamma)}\,\mathcal{I}_{0+}^{1-\,\gamma;\,\Psi }\left(\Psi(t)-\Psi(0)\right)^{k\alpha+\gamma-1}+\sum_{k=0}^{\infty}M^k\,\mathcal{I}_{0+}^{1-\,\gamma;\,\Psi }\mathcal{I}_{0+}^{\alpha(k+1); \Psi } g(t)\\
&=\lambda\sum_{k=0}^{\infty}\frac{M^k}{\Gamma(k\alpha+\gamma)}\frac{\Gamma(k\alpha+\gamma)}{\Gamma(k\alpha+\gamma+1-\,\gamma)}\left(\Psi(t)-\Psi(0)\right)^{k\alpha}+\sum_{k=0}^{\infty}M^k\,\mathcal{I}_{0+}^{1-\,\gamma+\alpha(k+1); \Psi } g(t)\\
&=\lambda\sum_{k=0}^{\infty}\frac{M^k}{\Gamma(k\alpha+1)}\left(\Psi(t)-\Psi(0)\right)^{k\alpha}+\sum_{k=0}^{\infty}M^k\,\mathcal{I}_{0+}^{1-\,\gamma+\alpha(k+1); \Psi } g(t)\\
&=\lambda\, E_{\alpha, 1}\left( M\left(\Psi(t)-\Psi(0)\right)^{\alpha}\right) \\
&\qquad+\int_{0}^{t} \Psi'(s)  (\Psi(t)-\Psi(s))^{\alpha-\gamma}  \sum_{k=0}^{\infty}\frac{\left(M (\Psi(t)-\Psi(s))^\alpha \right)^k}{\Gamma(k\alpha+\alpha+1-\,\gamma)}g(s) ds\\
&=\lambda\, E_{\alpha, 1}\left( M\left(\Psi(t)-\Psi(0)\right)^{\alpha}\right) \\
&\qquad+\int_{0}^{t} \Psi'(s)  (\Psi(t)-\Psi(s))^{\alpha-\gamma}  E_{\alpha,\,\alpha+1-\,\gamma}\left( M\left(\Psi(t)-\Psi(s)\right)^{\alpha}\right) g(s) ds,\, t\in (0,\,T].
\end{align*}
Therefore,
\begin{align}\label{eq534}
&\mathcal{I}_{0+}^{1-\,\gamma;\,\Psi }y(t, \lambda)|_{t=T} \nonumber \\
&=\lambda\, E_{\alpha, 1}\left( M\left(\Psi(T)-\Psi(0)\right)^{\alpha}\right) \nonumber\\
&\qquad+\int_{0}^{T} \Psi'(s)  (\Psi(T)-\Psi(s))^{\alpha-\gamma}  E_{\alpha,\,\alpha+1-\,\gamma}\left( M\left(\Psi(T)-\Psi(s)\right)^{\alpha}\right) g(s) ds.
\end{align}
Using equation \eqref{eq534} in equation \eqref{eq522a}, we get
\begin{align*}
g(\lambda)&=r\left\lbrace \lambda\, E_{\alpha, 1}\left( M\left(\Psi(T)-\Psi(0)\right)^{\alpha}\right)\right. \\
&\qquad\left.+\int_{0}^{T} \Psi'(s)  (\Psi(T)-\Psi(s))^{\alpha-\gamma}  E_{\alpha,\,\alpha+1-\,\gamma}\left( M\left(\Psi(T)-\Psi(s)\right)^{\alpha}\right) g(s) ds\right\rbrace - \lambda.
\end{align*}
Differentiating above equation with respect to $\lambda$ and using the condition on $r$, we obtain
$$
g'(\lambda)=r\, E_{\alpha, 1}\left( M\left(\Psi(T)-\Psi(0)\right)^{\alpha}\right)-1<0.
$$
This implies $g$ is strictly decreasing function on the closed interval $\left[\mathcal{I}_{0+}^{1-\,\gamma;\,\Psi }p(T),\,\mathcal{I}_{0+}^{1-\,\gamma;\,\Psi }q(T) \right]$. Therefore, we have
$$
g\left( \mathcal{I}_{0+}^{1-\,\gamma;\,\Psi }p(T)\right) >g\left( \mathcal{I}_{0+}^{1-\,\gamma;\,\Psi }q(T) \right).
$$
 Next, we show that the equation $g(\lambda)=0$ has atmost one solution on $\R$. Using equation \eqref{eq522b},  we obtain
\begin{align}\label{ak1}
g\left( \mathcal{I}_{0+}^{1-\,\gamma;\,\Psi }q(T)\right) =\left[ r\,\mathcal{I}_{0+}^{1-\,\gamma;\,\Psi } y(T,\, \mathcal{I}_{0+}^{1-\,\gamma;\,\Psi }q(T))- \mathcal{I}_{0+}^{1-\,\gamma;\,\Psi }q(T) \right] \leq 0,
\end{align}
and 
\begin{align}\label{ak2}
g\left( \mathcal{I}_{0+}^{1-\,\gamma;\,\Psi }p(T)\right)=\left[ r\,\mathcal{I}_{0+}^{1-\,\gamma;\,\Psi } y(T,\, \mathcal{I}_{0+}^{1-\,\gamma;\,\Psi }p(T))- \mathcal{I}_{0+}^{1-\,\gamma;\,\Psi }p(T) \right] \geq 0.
\end{align}
Since $g$ is continuous on the closed interval $\left[\mathcal{I}_{0+}^{1-\,\gamma;\,\Psi }p(T),\,\mathcal{I}_{0+}^{1-\,\gamma;\,\Psi }q(T) \right]$ and satisfies the conditions \eqref{ak1} and \eqref{ak2}, using the location of root theorem given in Lemma \ref{lin3} coupled with strictly decreasing nature of $g$, there exist at most one  $\lambda_0\in\left[\mathcal{I}_{0+}^{1-\,\gamma;\,\Psi }p(T),\,\mathcal{I}_{0+}^{1-\,\gamma;\,\Psi }q(T) \right]$ such that
\begin{equation}\label{ak3}
g(\lambda_0)=0.
\end{equation}
From equations \eqref{eq531}, \eqref{eq522a} and \eqref{ak3}, we obtain
\begin{equation}\label{eq5111}
r\,\mathcal{I}_{0+}^{1-\,\gamma;\,\Psi }y(t, \lambda_0)|_{t=T}=\lambda_0=\mathcal{I}_{0+}^{1-\,\gamma;\,\Psi }y(t)|_{t=0}.
\end{equation}
From equations \eqref{eq531} and \eqref{eq531a}, it follows that 
\begin{align}\label{eq5112}
y(t, \lambda_0)
&=\lambda_0\, (\Psi(t)-\Psi(0))^{\gamma-1}E_{\alpha,\,\gamma}\left(M (\Psi(t)-\Psi(0))^\alpha \right)\nonumber\\
&\qquad+\int_{0}^{t} \Psi'(s)  (\Psi(t)-\Psi(s))^{\alpha-1}  E_{\alpha,\,\alpha}\left(M (\Psi(t)-\Psi(s))^\alpha \right) g(s) ds
\end{align}
is the unique solution of linear   $\Psi$-Hilfer FDEs with  initial condition
\begin{align}\label{am551}
\begin{cases}
&^H \mathcal{D}^{\alpha,\,\beta\,;\, \Psi}_{0 +} y(t)-M y(t)=g(t),\, t\in  (0,\,T],\\
&\mathcal{I}_{0+}^{1-\,\gamma;\,\Psi } y(t)|_{t=0}= \lambda_0.
\end{cases}
\end{align}
Further, from \eqref{eq5111} and \eqref{am551}  it follows that, $y(t, \lambda_0)$ given in \eqref{eq5112} is the unique solution of linear $\Psi$-Hilfer BVP  \eqref{eq513}-\eqref{eq514}.\\
\textbf{Part II:} In this part we prove that the unique solution $y(t, \lambda_0)$ obtained in the first part satisfies 
$$
u\preceq y(\cdot, \lambda_0)\preceq v~ \text{ in}~ C_{1-\,\gamma ;\, \Psi }\left( J,\,\R\right).
$$
Define $h(t)=u(t)-y(t, \lambda_0),\, t\in (0,\,T]$. Clearly $h\in C_{1-\,\gamma ;\, \Psi }\left( J,\,\R\right)$.\\
Case 1: If $ r\,\mathcal{I}_{0+}^{1-\,\gamma;\,\Psi } u(T)\geq  \mathcal{I}_{0+}^{1-\,\gamma;\,\Psi } u(0)$, then $a_u(t)=0,\,t\in(0,\,T]$.  Since $u$ is a lower solution of the linear $\Psi$-Hilfer BVP  \eqref{eq513}-\eqref{eq514} and $y(t, \lambda_0)$ is the solution of linear $\Psi$-Hilfer FDEs   \eqref{am551}, we obtain
\begin{align*}
^H \mathcal{D}^{\alpha,\,\beta\,;\, \Psi}_{0 +} h(t)-M\, h(t)&=\,^H \mathcal{D}^{\alpha,\,\beta\,;\, \Psi}_{0 +} \left[u(t)- y(t, \lambda_0)\right] -M\,\left[ u(t)- y(t, \lambda_0)\right]\\
&=\left[ ^H \mathcal{D}^{\alpha,\,\beta\,;\, \Psi}_{0 +}  u(t) - M\,u(t)\right] - \left[ ^H \mathcal{D}^{\alpha,\,\beta\,;\, \Psi}_{0 +} y(t, \lambda_0) - M\,y(t, \lambda_0)\right] \\
&\leq \,g(t)-a_u(t)-  g(t)\\
&=0, \, t\in  (0,\,T].
\end{align*}
Further, using the equations \eqref{eq525} and \eqref{eq526} and initial condition in \eqref{am551}, we have 
\begin{align*}
\mathcal{I}_{0+}^{1-\,\gamma;\,\Psi } h(0)
&=\mathcal{I}_{0+}^{1-\,\gamma;\,\Psi } u(0)- \mathcal{I}_{0+}^{1-\,\gamma;\,\Psi } y(t, \lambda_0)|_{t=0}\\
&=\frac{\mathcal{I}_{0+}^{1-\,\gamma;\,\Psi }p(0)}{r}- \mathcal{I}_{0+}^{1-\,\gamma;\,\Psi } y(t, \lambda_0)|_{t=0}\\
&\leq \, \mathcal{I}_{0+}^{1-\,\gamma;\,\Psi }p(T)- \,\lambda_0\\
&\leq 0.
\end{align*}
Applying Theorem \ref{lem25} to the fractional inequalities 
\begin{align*}
\begin{cases}
&^H \mathcal{D}^{\alpha,\,\beta\,;\, \Psi}_{0 +} h(t)-M\, h(t)\leq 0, \, t\in  (0,\,T],\\
&\mathcal{I}_{0+}^{1-\,\gamma;\,\Psi } h(0) \leq 0,
\end{cases}
\end{align*}
we obtain $h(t)\leq 0,\,t\in (0,\,T].$ This gives, $u(t)\leq y(t, \lambda_0),\,t\in (0,\,T].$ \\
Case 2: If $ r\,\mathcal{I}_{0+}^{1-\,\gamma;\,\Psi } u(T)< \mathcal{I}_{0+}^{1-\,\gamma;\,\Psi } u(0)$ then $a_u(t)\neq0,\,t\in (0,\,T]$ as defined in \eqref{aa2.6}. Since $u$ is a lower solution of the linear $\Psi$-Hilfer  BVP  \eqref{eq513}-\eqref{eq514} and $y(t, \lambda_0)$ is the solution of linear $\Psi$-Hilfer FDEs   \eqref{am551}, we obtain
\begin{align*}
^H \mathcal{D}^{\alpha,\,\beta\,;\, \Psi}_{0 +} h(t)-M\, h(t)&=\,^H \mathcal{D}^{\alpha,\,\beta\,;\, \Psi}_{0 +} \left[u(t)- y(t, \lambda_0)\right] -M\,\left[ u(t)- y(t, \lambda_0)\right]\\
&=\left[ ^H \mathcal{D}^{\alpha,\,\beta\,;\, \Psi}_{0 +}  u(t) - M\,u(t)\right] - \left[ ^H \mathcal{D}^{\alpha,\,\beta\,;\, \Psi}_{0 +} y(t, \lambda_0) - M\,y(t, \lambda_0)\right] \\
&\leq \,g(t)-a_u(t)-  g(t)\\
&=-a_u(t), \, t\in  (0,\,T].
\end{align*}
Further, we have
$
\mathcal{I}_{0+}^{1-\,\gamma;\,\Psi } h(0)\leq 0.
$
By applying Theorem \ref{lem26} to the fractional inequalities 
\begin{align*}
\begin{cases}
&^H \mathcal{D}^{\alpha,\,\beta\,;\, \Psi}_{0 +} h(t)-M\, h(t)\leq \,g(t)-a_u(t)-  g(t)=-a_u(t), \, t\in  (0,\,T],\\
&\mathcal{I}_{0+}^{1-\,\gamma;\,\Psi } h(0)\leq 0,
\end{cases}
\end{align*}
we have $h(t)\leq 0,\,t\in (0,\,T].$ This implies, $u(t)\leq y(t, \lambda_0),\,t\in (0,\,T].$ From Case 1 and Case 2, it follows that 
$$
u(t)\leq y(t, \lambda_0),\,t\in (0,\,T].
$$ 
Following the similar approach, one can show that 
$$
v(t)\geq  y(t, \lambda_0),\,t\in (0,\,T].
$$
Therefore, we have
$$
u(t)\leq y(t, \lambda_0)\leq v(t),\,t\in (0,\,T].
$$
By Part I and Part II, it follows that $y(\cdot, \lambda_0)\in  C_{1-\,\gamma ;\, \Psi }\left( J,\,\R\right)$ is the unique solution  of the linear $\Psi$-Hilfer  BVP \eqref{eq513}-\eqref{eq514}  that satisfies the condition
$$
u\preceq y(\cdot, \lambda_0)\preceq v~ \text{ in}~ C_{1-\,\gamma ;\, \Psi }\left( J,\,\R\right).
$$
This completes the proof.
 \end{proof}
\section{Existence and uniqueness for the nonlinear $\Psi$-Hilfer BVP}
\begin{definition}
 We say that $w_0\in C_{1-\,\gamma ;\, \Psi }\left( J,\,\R\right)$ and $ z_0\in C_{1-\,\gamma ;\, \Psi }\left( J,\,\R\right)$ are the lower  and  upper solutions respectively of the nonlinear $\Psi$-Hilfer BVP \eqref{eq511}-\eqref{eq512}  if 
$$
^H \mathcal{D}^{\alpha,\,\beta\,;\, \Psi}_{0 +}w_0(t)- M w_0(t)\leq f(t, w_0(t))  - a_{w_0}(t), ~t \in  (0,\,T],
$$
and
$$
    ^H \mathcal{D}^{\alpha,\,\beta\,;\, \Psi}_{0 +}z_0(t)- M z_0(t) \geq f(t, z_0(t)) + b_{z_0}(t), ~t \in  (0,\,T],
      $$
 where
 \begin{equation*}
     a_{w_0}(t)=
     \begin{cases}
   0, ~\text{if} \,\,  r\,\mathcal{I}_{0+}^{1-\,\gamma;\,\Psi } w_0(T)\geq  \mathcal{I}_{0+}^{1-\,\gamma;\,\Psi } w_0(0)\\
   \frac{1}{r}\left(  ^H \mathcal{D}^{\alpha,\,\beta\,;\, \Psi}_{0 +} \xi(t)- M\xi(t) \right) \left( \mathcal{I}_{0+}^{1-\,\gamma;\,\Psi } w_0(0)- r\,\mathcal{I}_{0+}^{1-\,\gamma;\,\Psi } w_0(T)\right),\\
   ~\text{if} \,\,r\,\mathcal{I}_{0+}^{1-\,\gamma;\,\Psi } w_0(T) <  \mathcal{I}_{0+}^{1-\,\gamma;\,\Psi } w_0(0),
     \end{cases}
       \end{equation*}
 and      
        \begin{equation*}
          b_{z_0}(t)=
          \begin{cases}
        0,  ~\text{if} \,\,    r\,\mathcal{I}_{0+}^{1-\,\gamma;\,\Psi } z_0(T)\leq  \mathcal{I}_{0+}^{1-\,\gamma;\,\Psi } z_0(0)\\
        \frac{1}{r}\left(  ^H \mathcal{D}^{\alpha,\,\beta\,;\, \Psi}_{0 +} \xi(t)- M\xi(t) \right) \left(r\,\mathcal{I}_{0+}^{1-\,\gamma;\,\Psi } z_0(T) - \mathcal{I}_{0+}^{1-\,\gamma;\,\Psi } z_0(0)  \right),\\
         ~\text{if} \,\, r\,\mathcal{I}_{0+}^{1-\,\gamma;\,\Psi } z_0(T) >  \mathcal{I}_{0+}^{1-\,\gamma;\,\Psi } z_0(0),
          \end{cases}
          \end{equation*}
          and $\xi$ is the function as defined in \eqref{aa2.6a}.
 \end{definition}
 
 \begin{theorem}\label{thm541}
 Assume $z_0\in C_{1-\,\gamma ;\, \Psi }\left( J,\,\R\right)$ and $ w_0\in C_{1-\,\gamma ;\, \Psi }\left( J,\,\R\right)$  are the upper and lower solutions respectively of the nonlinear $\Psi$-Hilfer  BVP \eqref{eq511}-\eqref{eq512}  such that $w_0\preceq z_0$.  Further, assume that:\\
 (H1) the function $f$ satisfies 
 \begin{enumerate}
 \item [(i)]  $f\left( \cdot, y(\cdot)\right) \in C_{1-\,\gamma ;\, \Psi }\left( J,\,\R\right)$ for each $y\in C_{1-\,\gamma ;\, \Psi }\left( J,\,\R\right)$, 
 \item [(ii)] 
$
f(t,\,y_1) \leq f(t,\,y_2),~ \text{for any } y_1,\,y_2\in \R \text{ with } y_1\leq y_2~~ \text{and}~~t\in (0,\,T]. 
$
 \end{enumerate}
Then, the nonlinear $\Psi$-Hilfer BVP \eqref{eq511}-\eqref{eq512} has a minimal solution $w^*\in C_{1-\,\gamma ;\, \Psi }\left( J,\,\R\right)$ and a maximal solution $z^*\in C_{1-\,\gamma ;\, \Psi }\left( J,\,\R\right)$ in the line segment
 $$
 [w_0,\,z_0]=\{y\in C_{1-\,\gamma ;\, \Psi }\left( J,\,\R\right):w_0 \preceq y \preceq z_0 \}.
 $$
 Further, if $\{w_n\}_{n=1}^\infty$  and $\{z_n\}_{n=1}^\infty$ are the iterative sequences defined by 
\begin{align*}
 w_n(t)&=\frac{r\,(\Psi(t)-\Psi(0))^{\gamma-1}E_{\alpha,\,\gamma}\left(M (\Psi(t)-\Psi(0))^\alpha \right) }{1-r\,E_{\alpha,\,1}\left(M (\Psi(T)-\Psi(0))^\alpha \right) }\\
 &\qquad \times \int_{0}^{T} \Psi'(s)  (\Psi(T)-\Psi(s))^{\alpha-\gamma}  E_{\alpha,\,\alpha+1-\,\gamma}\left(M (\Psi(T)-\Psi(s))^\alpha \right) f(s,\,w_{n-1}(s)) ds\\
 &\qquad+\int_{0}^{t} \Psi'(s)  (\Psi(t)-\Psi(s))^{\alpha-1}  E_{\alpha,\,\alpha}\left(M (\Psi(t)-\Psi(s))^\alpha \right) f(s,\,w_{n-1}(s)) ds
\end{align*}
and
\begin{align*}
 z_n(t)&=\frac{r\,(\Psi(t)-\Psi(0))^{\gamma-1}E_{\alpha,\,\gamma}\left(M (\Psi(t)-\Psi(0))^\alpha \right) }{1-r\,E_{\alpha,\,1}\left(M (\Psi(T)-\Psi(0))^\alpha \right) }\\
 &\qquad \times\int_{0}^{T} \Psi'(s)  (\Psi(T)-\Psi(s))^{\alpha-\gamma}  E_{\alpha,\,\alpha+1-\,\gamma}\left(M (\Psi(T)-\Psi(s))^\alpha \right) f(s,\,z_{n-1}(s)) ds\\
 &\qquad+\int_{0}^{t} \Psi'(s)  (\Psi(t)-\Psi(s))^{\alpha-1}  E_{\alpha,\,\alpha}\left(M (\Psi(t)-\Psi(s))^\alpha \right) f(s,\,z_{n-1}(s)) ds
 \end{align*}
 then $\{w_n\}_{n=1}^\infty$  and $\{z_n\}_{n=1}^\infty$ are the monotonic sequences in $C_{1-\,\gamma ;\, \Psi }\left( J,\,\R\right)$ such that
\begin{align*}
& \lim\limits_{n\rightarrow\infty}\left\Vert w_n-w^*\right\Vert _{C_{1-\,\gamma ;\Psi }\left(  J, \R\right)  }=0
\end{align*}
 and
\begin{align*}
 &\lim\limits_{n\rightarrow\infty}\left\Vert z_n-z^*\right\Vert _{C_{1-\,\gamma ;\Psi }\left(  J, \R\right)  }=0.
\end{align*}
 \end{theorem}
 \begin{proof}
 We  give  the   proof in five parts.\\
\textbf{Part 1:} We denote $\mathfrak{D}=[w_0,\,z_0]$. For any $\varphi\in \mathfrak{D}$, we consider the following linear $\Psi$-Hilfer BVP 
\begin{equation}\label{eq533}
\left.
\begin{aligned}
 & ^H \mathcal{D}^{\alpha,\,\beta\,;\, \Psi}_{0 +}y(t)- M y(t)
   = f(t, \varphi(t)), ~t \in  (0,\,T],  \\
  & \mathcal{I}_{0+}^{1-\,\gamma;\,\Psi } y(0)= r\,\mathcal{I}_{0+}^{1-\,\gamma;\,\Psi } y(T).
\end{aligned}
\right\}
\end{equation}
Since $w_0, z_0$ are lower and upper solutions respectively of the nonlinear $\Psi$-Hilfer  BVP  \eqref{eq511}-\eqref{eq512}, using hypothesis (H1)(ii), we have
\begin{align*}
^H \mathcal{D}^{\alpha,\,\beta\,;\, \Psi}_{0 +}w_0(t)- M w_0(t) \leq f(t, w_0(t))  - a_{w_0}(t) \leq f(t, \varphi(t))  - a_{w_0}(t), ~t \in  (0,\,T]
\end{align*}
and
\begin{align*}
^H \mathcal{D}^{\alpha,\,\beta\,;\, \Psi}_{0 +}z_0(t)- M z_0(t) \geq f(t, z_0(t))  + b_{z_0}(t) \geq f(t, \varphi(t))  + b_{z_0}(t), ~t \in  (0,\,T].
\end{align*}
This implies  $w_0$ and $z_0$ are the lower and upper solutions respectively of the linear $\Psi$-Hilfer  BVP  \eqref{eq533}. In the view of  Theorem \ref{lem531}, the linear $\Psi$-Hilfer BVP  \eqref{eq533} has a unique solution  $y\in \mathfrak{D}$, given by
\begin{align}\label{eq535}
y(t)&= \mathcal{I}_{0+}^{1-\,\gamma;\,\Psi } y(0)\, (\Psi(t)-\Psi(0))^{\gamma-1}E_{\alpha,\,\gamma}\left(M (\Psi(t)-\Psi(0))^\alpha \right)\nonumber\\
&\qquad+\int_{0}^{t} \Psi'(s)  (\Psi(t)-\Psi(s))^{\alpha-1}  E_{\alpha,\,\alpha}\left(M (\Psi(t)-\Psi(s))^\alpha \right) f(s, \varphi(s)) ds, ~t \in  (0,\,T].
\end{align}
Following the similar steps as  in the proof of the Theorem \ref{lem531}, from equation \eqref{eq535}, we have
\begin{align*}
\mathcal{I}_{0+}^{1-\,\gamma;\,\Psi }y(T)&=\mathcal{I}_{0+}^{1-\,\gamma;\,\Psi } y(0)\, E_{\alpha, 1}\left( M\left(\Psi(T)-\Psi(0)\right)^{\alpha}\right) \nonumber\\
&~+\int_{0}^{T} \Psi'(s)  (\Psi(T)-\Psi(s))^{\alpha-\gamma}  E_{\alpha,\,\alpha+1-\,\gamma}\left( M\left(\Psi(T)-\Psi(s)\right)^{\alpha}\right) f(s, \varphi(s))  ds.
\end{align*}
Since
$\mathcal{I}_{0+}^{1-\,\gamma;\,\Psi } y(0)= r\,\mathcal{I}_{0+}^{1-\,\gamma;\,\Psi } y(T)$, we have
\begin{align*}
\mathcal{I}_{0+}^{1-\,\gamma;\,\Psi } y(0)&= r\,\mathcal{I}_{0+}^{1-\,\gamma;\,\Psi } y(0)\, E_{\alpha, 1}\left( M\left(\Psi(T)-\Psi(0)\right)^{\alpha}\right) \\
&~+r\int_{0}^{T} \Psi'(s)  (\Psi(T)-\Psi(s))^{\alpha-\gamma}  E_{\alpha,\,\alpha+1-\,\gamma}\left( M\left(\Psi(T)-\Psi(s)\right)^{\alpha}\right) f(s, \varphi(s))  ds.
\end{align*}
Hence,
\begin{align}\label{eq536}
\mathcal{I}_{0+}^{1-\,\gamma;\,\Psi } y(0)&=\frac{r}{\left[ 1- r\, E_{\alpha, 1}\left( M\left(\Psi(T)-\Psi(0)\right)^{\alpha}\right) \right]} \nonumber\\
&\quad \times\int_{0}^{T} \Psi'(s)  (\Psi(T)-\Psi(s))^{\alpha-\gamma}  E_{\alpha,\,\alpha+1-\,\gamma}\left( M\left(\Psi(T)-\Psi(s)\right)^{\alpha}\right) f(s, \varphi(s))  ds.
\end{align}
Using equation \eqref{eq536} in \eqref{eq535}, we get
\begin{align*}
y(t)&=\frac{r\, (\Psi(t)-\Psi(0))^{\gamma-1}E_{\alpha,\,\gamma}\left(M (\Psi(t)-\Psi(0))^\alpha \right)}{\left[ 1- r\, E_{\alpha, 1}\left( M\left(\Psi(T)-\Psi(0)\right)^{\alpha}\right) \right]} \nonumber\\
&\qquad \times\int_{0}^{T} \Psi'(s)  (\Psi(T)-\Psi(s))^{\alpha-\gamma}  E_{\alpha,\,\alpha+1-\,\gamma}\left( M\left(\Psi(T)-\Psi(s)\right)^{\alpha}\right) f(s, \varphi(s))  ds\\
&\qquad+\int_{0}^{t} \Psi'(s)  (\Psi(t)-\Psi(s))^{\alpha-1}  E_{\alpha,\,\alpha}\left(M (\Psi(t)-\Psi(s))^\alpha \right) f(s, \varphi(s)) ds, ~t \in  (0,\,T].
\end{align*}
Consider the  operator $\mathcal{A}:\mathfrak{D}\rightarrow \mathcal{X}$ defined by 
\begin{align}\label{ie1}
\mathcal{A}\varphi(t)&=\frac{r\, (\Psi(t)-\Psi(0))^{\gamma-1}E_{\alpha,\,\gamma}\left(M (\Psi(t)-\Psi(0))^\alpha \right)}{\left[ 1- r\, E_{\alpha, 1}\left( M\left(\Psi(T)-\Psi(0)\right)^{\alpha}\right) \right]} \nonumber\\
&\qquad \times \int_{0}^{T} \Psi'(s)  (\Psi(T)-\Psi(s))^{\alpha-\gamma}  E_{\alpha,\,\alpha+1-\,\gamma}\left( M\left(\Psi(T)-\Psi(s)\right)^{\alpha}\right) f(s, \varphi(s))  ds\nonumber\\
&\quad+\int_{0}^{t} \Psi'(s)  (\Psi(t)-\Psi(s))^{\alpha-1}  E_{\alpha,\,\alpha}\left(M (\Psi(t)-\Psi(s))^\alpha \right) f(s, \varphi(s)) ds, ~t \in  (0,\,T].
\end{align}
It follows that, $\mathcal{A}\varphi(t)$ is the solution of the linear BVP \eqref{eq533}. Using the Theorem \ref{lem531}, we have 
$$
w_0(t)\leq \mathcal{A}\varphi(t)\leq z_0(t), ~\text{for}~\varphi\in \mathfrak{D}=[w_0,\,z_0] ~\text{and} ~t \in  (0,\,T].
$$
In particular, we have $w_0\preceq \mathcal{A}w_0$ and $\mathcal{A}z_0\preceq z_0$ in $C_{1-\,\gamma ;\, \Psi }\left( J,\,\R\right)$.\\
\textbf{ Part 2:} In this part, we prove that $\mathcal{A}:\mathfrak{D}\rightarrow \mathcal{X}$  is completely continuous operator. \\
Firstly we prove that, $\mathcal{A}(\mathfrak{D})$ is uniformly bounded. Since $(\Psi(\cdot)-\Psi(0))^{1-\gamma}f(\cdot, y(\cdot))$ is continuous on compact interval $J$,  it is  bounded. Hence, there exists constant $\mathfrak{K}>0$ such that
\begin{equation}\label{b1}
\underset{t\in J
 }{\max }\left\vert (\Psi(t)-\Psi(0))^{1-\gamma} f(t,\,y(t)) \right\vert\leq \mathfrak{K}.
\end{equation}
Using increasing nature of $\Psi$, condition \eqref{b1} and the Lemma \ref{lema2}(b), for any $t\in J,$ we obtain
\begin{align*}
&\left\vert(\Psi(t)-\Psi(0))^{1-\,\gamma} \mathcal{A}\varphi(t)\right\vert\\
&=\left\vert \frac{r\, E_{\alpha,\,\gamma}\left(M (\Psi(t)-\Psi(0))^\alpha \right)}{\left[ 1- r\, E_{\alpha, 1}\left( M\left(\Psi(T)-\Psi(0)\right)^{\alpha}\right) \right]} \nonumber\right.\\
& \left.\qquad \quad  \times \int_{0}^{T} \Psi'(s)  (\Psi(T)-\Psi(s))^{\alpha-\gamma}  E_{\alpha,\,\alpha+1-\,\gamma}\left( M\left(\Psi(T)-\Psi(s)\right)^{\alpha}\right) f(s, \varphi(s))  ds\right.\\
&\left.\qquad+(\Psi(t)-\Psi(0))^{1-\,\gamma}\int_{0}^{t} \Psi'(s)  (\Psi(t)-\Psi(s))^{\alpha-1}  E_{\alpha,\,\alpha}\left(M (\Psi(t)-\Psi(s))^\alpha \right) f(s, \varphi(s)) ds\right\vert\\
&\leq \frac{r\, E_{\alpha,\,\gamma}\left(M (\Psi(T)-\Psi(0))^\alpha \right)\,E_{\alpha,\,\alpha+1-\,\gamma}\left(M (\Psi(T)-\Psi(0))^\alpha \right)}{\left[ 1- r\, E_{\alpha, 1}\left( M\left(\Psi(T)-\Psi(0)\right)^{\alpha}\right) \right]}\\
&\qquad \quad \times \int_{0}^{T} \Psi'(s)  (\Psi(T)-\Psi(s))^{\alpha-\gamma}  (\Psi(s)-\Psi(0))^{\gamma-1} \left| (\Psi(s)-\Psi(0))^{1-\gamma} f(s, \varphi(s)) \right|  ds\\
&\quad+(\Psi(T)-\Psi(0))^{1-\,\gamma} E_{\alpha,\,\alpha}\left(M (\Psi(T)-\Psi(0))^\alpha \right)\\
&\qquad \quad \times \int_{0}^{t} \Psi'(s)  (\Psi(t)-\Psi(s))^{\alpha-1} (\Psi(s)-\Psi(0))^{\gamma-1}  \left| (\Psi(s)-\Psi(0))^{1-\gamma} f(s, \varphi(s)) \right|  ds\\
&\leq \frac{r\, E_{\alpha,\,\gamma}\left(M (\Psi(T)-\Psi(0))^\alpha \right)\,E_{\alpha,\,\alpha+1-\,\gamma}\left(M (\Psi(T)-\Psi(0))^\alpha \right)}{\left[ 1- r\, E_{\alpha, 1}\left( M\left(\Psi(T)-\Psi(0)\right)^{\alpha}\right) \right]}\\
&\qquad \quad \times \mathfrak{K}\,\Gamma(\alpha-\gamma+1)\left[\mathcal{I}_{0+}^{\alpha-\,\gamma+1;\,\Psi }   (\Psi(t)-\Psi(0))^{\gamma-1}  \right] _{t=T}   \\
&+(\Psi(T)-\Psi(0))^{1-\,\gamma} E_{\alpha,\,\alpha}\left(M (\Psi(T)-\Psi(0))^\alpha \right)\, \mathfrak{K}\, \Gamma(\alpha)\, \mathcal{I}_{0+}^{\alpha;\,\Psi }   (\Psi(t)-\Psi(0))^{\gamma-1} \\
&\leq \frac{r\, E_{\alpha,\,\gamma}\left(M (\Psi(T)-\Psi(0))^\alpha \right)\,E_{\alpha,\,\alpha+1-\,\gamma}\left(M (\Psi(T)-\Psi(0))^\alpha \right)}{\left[ 1- r\, E_{\alpha, 1}\left( M\left(\Psi(T)-\Psi(0)\right)^{\alpha}\right) \right]}\\
&\qquad \quad \times \mathfrak{K}\,\Gamma(\alpha-\gamma+1) \frac{\Gamma(\gamma)}{\Gamma(\alpha+1)} (\Psi(T)-\Psi(0))^{\alpha}  \\
&+(\Psi(T)-\Psi(0))^{1-\,\gamma} E_{\alpha,\,\alpha}\left(M (\Psi(T)-\Psi(0))^\alpha \right)\, \mathfrak{K}\, \Gamma(\alpha) \frac{\Gamma(\gamma)}{\Gamma(\gamma+\alpha)} (\Psi(T)-\Psi(0))^{\alpha+\gamma-1} \\
&= \mathfrak{K}(\Psi(T)-\Psi(0))^{\alpha}\\ 
&\times \quad \left\lbrace B\left(\alpha-\gamma+1, \gamma \right)\frac{r\, E_{\alpha,\,\gamma}\left(M (\Psi(T)-\Psi(0))^\alpha \right)\,E_{\alpha,\,\alpha+1-\,\gamma}\left(M (\Psi(T)-\Psi(0))^\alpha \right)}{\left[ 1- r\, E_{\alpha, 1}\left( M\left(\Psi(T)-\Psi(0)\right)^{\alpha}\right) \right]}\right.\\
& \quad+ B\left(\alpha, \gamma \right) E_{\alpha,\,\alpha}\left(M (\Psi(T)-\Psi(0))^\alpha \right)\Big\} \\
&:=\omega\end{align*}
From above inequality it follows that, 
$$
\left\Vert \mathcal{A}\varphi\right\Vert _{C_{1-\,\gamma ;\Psi }\left( J,\,\R\right) }\leq \omega,~\varphi\in\mathfrak{D}.
$$
This implies, $\mathcal{A}(\mathfrak{D})$ is uniformly bounded. Next, we prove that  $\mathcal{A}(\mathfrak{D})$ is equicontinuous. Let any $t_1, t_2\in J$ such that $t_2>t_1$. Then, for any $\varphi\in \mathfrak{D}$, we obtain
\begin{align}\label{equi3}
&\left| \left( \Psi(t_2)-\Psi(0)\right)^{1-\,\gamma} \mathcal{A}\varphi(t_2)-\left( \Psi(t_1)-\Psi(0)\right)^{1-\,\gamma} \mathcal{A}\varphi(t_1)\right|\nonumber\\
&=\left|\left\lbrace \frac{r\, E_{\alpha,\,\gamma}\left(M (\Psi(t_2)-\Psi(0))^\alpha \right)}{\left[ 1- r\, E_{\alpha, 1}\left( M\left(\Psi(T)-\Psi(0)\right)^{\alpha}\right) \right]} \nonumber\right.\right.\nonumber\\
&\left.\left.\qquad \times \int_{0}^{T} \Psi'(s)  (\Psi(T)-\Psi(s))^{\alpha-\gamma}  E_{\alpha,\,\alpha+1-\,\gamma}\left( M\left(\Psi(T)-\Psi(s)\right)^{\alpha}\right) f(s, \varphi(s))  ds\right.\right.\nonumber\\
&\left.\left.\qquad+(\Psi(t_2)-\Psi(0))^{1-\,\gamma}\int_{0}^{t_2} \Psi'(s)  (\Psi(t_2)-\Psi(s))^{\alpha-1}  E_{\alpha,\,\alpha}\left(M (\Psi(t_2)-\Psi(s))^\alpha \right) f(s, \varphi(s)) ds \right\rbrace\right.\nonumber \\
&\left.\quad-\left\lbrace \frac{r\, E_{\alpha,\,\gamma}\left(M (\Psi(t_1)-\Psi(0))^\alpha \right)}{\left[ 1- r\, E_{\alpha, 1}\left( M\left(\Psi(T)-\Psi(0)\right)^{\alpha}\right) \right]} \nonumber\right.\right.\nonumber\\
&\left.\left.\qquad \times\int_{0}^{T} \Psi'(s)  (\Psi(T)-\Psi(s))^{\alpha-\gamma}  E_{\alpha,\,\alpha+1-\,\gamma}\left( M\left(\Psi(T)-\Psi(s)\right)^{\alpha}\right) f(s, \varphi(s))  ds\right.\right.\nonumber\\
&\left.\left.\qquad+(\Psi(t_1)-\Psi(0))^{1-\,\gamma}\int_{0}^{t_1} \Psi'(s)  (\Psi(t_1)-\Psi(s))^{\alpha-1}  E_{\alpha,\,\alpha}\left(M (\Psi(t_1)-\Psi(s))^\alpha \right) f(s, \varphi(s)) ds \right\rbrace \right| \nonumber \\
&\leq \frac{r \, E_{\alpha,\,\alpha+1-\,\gamma}\left( M\left(\Psi(T)-\Psi(0)\right)^{\alpha}\right)}{\left[ 1- r\, E_{\alpha, 1}\left( M\left(\Psi(T)-\Psi(0)\right)^{\alpha}\right) \right]} \left| E_{\alpha,\,\gamma}\left(M (\Psi(t_2)-\Psi(0))^\alpha\right) -E_{\alpha,\,\gamma}\left(M (\Psi(t_1)-\Psi(0))^\alpha\right)  \right|  \nonumber\\
&\qquad \times \int_{0}^{T} \Psi'(s)  (\Psi(T)-\Psi(s))^{\alpha-\gamma}  (\Psi(s)-\Psi(0))^{\gamma-1}  \left| (\Psi(s)-\Psi(0))^{1-\gamma}  f(s, \varphi(s)) \right|  ds\nonumber\\
&\quad+\left| (\Psi(t_2)-\Psi(0))^{1-\,\gamma}E_{\alpha,\,\alpha}\left(M (\Psi(T)-\Psi(0))^\alpha \right)\right.\nonumber\\
&\left.\quad \times\int_{0}^{t_2} \Psi'(s)  (\Psi(t_2)-\Psi(s))^{\alpha-1}   (\Psi(s)-\Psi(0))^{\gamma-1}  \left| (\Psi(s)-\Psi(0))^{1-\gamma}  f(s, \varphi(s))\right|  ds\right.\nonumber\\
&\quad-\left. (\Psi(t_1)-\Psi(0))^{1-\,\gamma}E_{\alpha,\,\alpha}\left(M (\Psi(T)-\Psi(0))^\alpha \right)\right.\nonumber\\
&\left.\quad \times \int_{0}^{t_1} \Psi'(s)  (\Psi(t_1)-\Psi(s))^{\alpha-1}  (\Psi(s)-\Psi(0))^{\gamma-1}  \left| (\Psi(s)-\Psi(0))^{1-\gamma}  f(s, \varphi(s))\right|  ds\right|\nonumber \\
&\leq \frac{r \, E_{\alpha,\,\alpha+1-\,\gamma}\left( M\left(\Psi(T)-\Psi(0)\right)^{\alpha}\right)}{\left[ 1- r\, E_{\alpha, 1}\left( M\left(\Psi(T)-\Psi(0)\right)^{\alpha}\right) \right]} \left| E_{\alpha,\,\gamma}\left(M (\Psi(t_2)-\Psi(0))^\alpha\right) -E_{\alpha,\,\gamma}\left(M (\Psi(t_1)-\Psi(0))^\alpha\right)  \right|  \nonumber\\
&\qquad \times \mathfrak{K} \Gamma(\alpha-\gamma+1)\left[\mathcal{I}_{0+}^{\alpha-\,\gamma+1;\,\Psi }   (\Psi(t)-\Psi(0))^{\gamma-1}  \right] _{t=T} \nonumber \\
&\quad+\left|\mathfrak{K} (\Psi(t_2)-\Psi(0))^{1-\,\gamma}E_{\alpha,\,\alpha}\left(M (\Psi(T)-\Psi(0))^\alpha \right)\Gamma(\alpha)\left[\mathcal{I}_{0+}^{\alpha;\,\Psi }   (\Psi(t)-\Psi(0))^{\gamma-1}  \right] _{t=t_2}\right.\nonumber\\
&\quad-\left.\mathfrak{K} (\Psi(t_1)-\Psi(0))^{1-\,\gamma}E_{\alpha,\,\alpha}\left(M (\Psi(T)-\Psi(0))^\alpha \right)\Gamma(\alpha)\left[\mathcal{I}_{0+}^{\alpha;\,\Psi }   (\Psi(t)-\Psi(0))^{\gamma-1}  \right] _{t=t_1}\right| \nonumber\\
&\leq \frac{\mathfrak{K}\,r\,E_{\alpha,\,\alpha+1-\,\gamma}\left( M\left(\Psi(T)-\Psi(0)\right)^{\alpha}\right)}{\left[ 1- r\, E_{\alpha, 1}\left( M\left(\Psi(T)-\Psi(0)\right)^{\alpha}\right) \right]} B(\alpha-\gamma+1,\gamma) (\Psi(T)-\Psi(0))^{\alpha}  \nonumber\\
&\qquad \times \left| E_{\alpha,\,\gamma}\left(M (\Psi(t_2)-\Psi(0))^\alpha\right) -E_{\alpha,\,\gamma}\left(M (\Psi(t_1)-\Psi(0))^\alpha\right)  \right|\nonumber\\
&\qquad + \mathfrak{K}\,E_{\alpha,\,\alpha}\left(M (\Psi(T)-\Psi(0))^\alpha \right)B(\alpha,\gamma)\left| (\Psi(t_2)-\Psi(0))^{\alpha}-(\Psi(t_1)-\Psi(0))^{\alpha} \right|. 
\end{align}
By Lemma \ref{lem28}, two parameter Mittag-Leffler function is uniformly continuous. Therefore, we have
\begin{equation}\label{equi1}
\left| E_{\alpha,\,\gamma}\left(M (\Psi(t_2)-\Psi(0))^\alpha\right) -E_{\alpha,\,\gamma}\left(M (\Psi(t_1)-\Psi(0))^\alpha\right)  \right|\rightarrow0~\text{as}~\left| t_2-t_1 \right|\rightarrow0.
\end{equation}
Further, using the continuity of $\Psi$, we have 
\begin{equation}\label{equi2}
\left| (\Psi(t_2)-\Psi(0))^{\alpha}-(\Psi(t_1)-\Psi(0))^{\alpha} \right|\rightarrow0~\text{as}~\left| t_2-t_1 \right|\rightarrow0.
\end{equation}
Using the conditions \eqref{equi1} and \eqref{equi2} in the inequality \eqref{equi3}, we obtain
$$
\left| \left( \Psi(t_2)-\Psi(0)\right)^{1-\,\gamma} \mathcal{A}\varphi(t_2)-\left( \Psi(t_1)-\Psi(0)\right)^{1-\,\gamma} \mathcal{A}\varphi(t_1)\right|\rightarrow0~\text{as}~\left| t_2-t_1 \right|\rightarrow0.
$$
This proves $\mathcal{A}(\mathfrak{D})$ is equicontinuous set of family of functions. Therefore, by Arzel$\grave{a}$-Ascoli theorem, $\mathcal{A}(\mathfrak{D})$ is relatively compact. Note that, the continuity of operator $\mathcal{A}$ follows from hypothesis (H1)(i). We have proved  $\mathcal{A}:\mathfrak{D}\rightarrow \mathcal{X}$   is completely continuous operator.\\ 
\textbf{Part 3:} In this part, it is proved that $\mathcal{A}:\mathfrak{D}\rightarrow \mathcal{X}$ is monotonically increasing operator.\\
Let any $\delta_1, \delta_2\in \mathfrak{D}$  with $w_0\leq \delta_1\leq \delta_2\leq z_0$. Define $\mathcal{B}(t)=f(t,\,\delta_2(t))-f(t,\,\delta_1(t)),\,t\in (0, T]$. Then from hypothesis (H1)(ii), we have $\mathcal{B}(t)\geq 0, ~t\in (0, T]$. For any $t\in (0, T]$, 
\begin{align*}
&\mathcal{A}\delta_2(t)-\mathcal{A}\delta_1(t)\\
&=\left\lbrace \frac{r\, (\Psi(t)-\Psi(0))^{\gamma-1}E_{\alpha,\,\gamma}\left(M (\Psi(t)-\Psi(0))^\alpha \right)}{\left[ 1- r\, E_{\alpha, 1}\left( M\left(\Psi(T)-\Psi(0)\right)^{\alpha}\right) \right]} \nonumber\right.\\
&\left.\qquad \times\int_{0}^{T} \Psi'(s)  (\Psi(T)-\Psi(s))^{\alpha-\gamma}  E_{\alpha,\,\alpha+1-\,\gamma}\left( M\left(\Psi(T)-\Psi(s)\right)^{\alpha}\right) f(s, \delta_2(s))  ds\right.\\
&\left.\qquad+\int_{0}^{t} \Psi'(s)  (\Psi(t)-\Psi(s))^{\alpha-1}  E_{\alpha,\,\alpha}\left(M (\Psi(t)-\Psi(s))^\alpha \right) f(s, \delta_2(s)) ds\right\rbrace \\
&-\left\lbrace \frac{r\, (\Psi(t)-\Psi(0))^{\gamma-1}E_{\alpha,\,\gamma}\left(M (\Psi(t)-\Psi(0))^\alpha \right)}{\left[ 1- r\, E_{\alpha, 1}\left( M\left(\Psi(T)-\Psi(0)\right)^{\alpha}\right) \right]} \nonumber\right.\\
&\left.\qquad \times\int_{0}^{T} \Psi'(s)  (\Psi(T)-\Psi(s))^{\alpha-\gamma}  E_{\alpha,\,\alpha+1-\,\gamma}\left( M\left(\Psi(T)-\Psi(s)\right)^{\alpha}\right) f(s, \delta_1(s))  ds\right.\\
&\left.\qquad+\int_{0}^{t} \Psi'(s)  (\Psi(t)-\Psi(s))^{\alpha-1}  E_{\alpha,\,\alpha}\left(M (\Psi(t)-\Psi(s))^\alpha \right) f(s, \delta_1(s)) ds\right\rbrace \\
&= \frac{r\, (\Psi(t)-\Psi(0))^{\gamma-1}E_{\alpha,\,\gamma}\left(M (\Psi(t)-\Psi(0))^\alpha \right)}{\left[ 1- r\, E_{\alpha, 1}\left( M\left(\Psi(T)-\Psi(0)\right)^{\alpha}\right) \right]} \nonumber\\
&\qquad \times \int_{0}^{T} \Psi'(s)  (\Psi(T)-\Psi(s))^{\alpha-\gamma}  E_{\alpha,\,\alpha+1-\,\gamma}\left( M\left(\Psi(T)-\Psi(s)\right)^{\alpha}\right) \mathcal{B}(s) ds\\
&\qquad+\int_{0}^{t} \Psi'(s)  (\Psi(t)-\Psi(s))^{\alpha-1}  E_{\alpha,\,\alpha}\left(M (\Psi(t)-\Psi(s))^\alpha \right)  \mathcal{B}(s) ds.
\end{align*}
Since $\mathcal{B}(t)\geq 0, ~t\in (0, T]$ and $\Psi$ is increasing function on $J$, the integrands of both the integrals in the right hand side of the above inequality is non-negative on $ (0, T]$. Therefore, we have $\mathcal{A}\delta_2(t)-\mathcal{A}\delta_1(t)\geq 0, ~t\in (0, T]$. This gives $\mathcal{A}\delta_2\succeq\mathcal{A}\delta_1$. Therefore, $\mathcal{A}:\mathfrak{D}\rightarrow \mathcal{X}$ is monotonically increasing operator.\\
\textbf{Part 4:} For each $n (n=1, 2, 3,\cdots) $ define $w_n=\mathcal{A}\,w_{n-1}$  and $z_n=\mathcal{A}\,z_{n-1}$. By Part 1, it follows that
$$
w_0\preceq\mathcal{A}w_0=w_1~\text{and}~ z_1=\mathcal{A}z_0\preceq z_0 \text{ in}~C_{1-\,\gamma ;\, \Psi }\left( J,\,\R\right).
$$
Therefore, by using increasing nature of an  operator $\mathcal{A}$,  we have
$$
 w_1\preceq w_2\preceq\cdots\preceq w_n\preceq z_n\preceq \cdots\preceq z_2\preceq z_1 \text{ in}~C_{1-\,\gamma ;\, \Psi }\left( J,\,\R\right).
$$
This implies $\{w_n\}_{n=1}^\infty$ and $ \{z_n\}_{n=1}^\infty$ are the monotonic sequences in  $\mathcal{A}(\mathfrak{D})\subseteq \mathcal{X}$ which are relatively compact also. Therefore,  by applying Lemma \ref{lem29}, there exists $w^*, z^*\in C_{1-\,\gamma ;\, \Psi }\left( J,\,\R\right)$   such that
\begin{equation}\label{inc1}
 \lim\limits_{n\rightarrow\infty}\left\Vert w_n-w^*\right\Vert _{C_{1-\,\gamma ;\Psi }\left(  J, \R\right)  }=0
\end{equation}
and
	\begin{equation}\label{inc2}
\lim\limits_{n\rightarrow\infty}\left\Vert z_n-z^*\right\Vert _{C_{1-\,\gamma ;\Psi }\left(  J, \R\right)  }=0.	
	\end{equation}
Since for each $n$,  $w_n=\mathcal{A}\,w_{n-1}$  and $z_n=\mathcal{A}\,z_{n-1}$, by   continuity of the operator $\mathcal{A}$ and  using the limits \eqref{inc1} and \eqref{inc2},
we obtain 
$$
w^*=\mathcal{A}w^*~\text{and}~ z^*=\mathcal{A}z^*.
$$
Therefore, $w^*$ and $z^*$ are the fixed points of an operator $\mathcal{A}$. Further, we know that, $y\in C_{1-\,\gamma ;\, \Psi }\left( J,\,\R\right)$ is the solution of the nonlinear $\Psi$-Hilfer BVP \eqref{eq511}-\eqref{eq512} if and only if it is the fixed point of an operator $\mathcal{A}$. Thus, $w^*$ and $z^*$ are the solutions of the nonlinear  $\Psi$-Hilfer BVP \eqref{eq511}-\eqref{eq512}.\\
\textbf{Part 5:} Finally, we prove that  $w^*$ and $z^*$ are the  minimal  and the maximal solutions respectively of the nonlinear $\Psi$-Hilfer BVP \eqref{eq511}-\eqref{eq512}.
Let  $y\in [w_0,\,z_0]$ be any solution of  the $\Psi$-Hilfer BVP \eqref{eq511}-\eqref{eq512}. Then, 
$$
y=\mathcal{A}y ~~\text{and}~~  w_0\preceq y \preceq z_0.
$$
Since  $\mathcal{A}$ is an increasing  operator, we have
$$
w_1=\mathcal{A}w_0\preceq \mathcal{A}y=y\preceq \mathcal{A}z_0=z_1.
$$
Again, using the increasing nature of an operator $\mathcal{A}$, from above inequality,  we obtain
$$
w_2\preceq y\preceq z_2.
$$
Continuing in this way, we obtain
$$
w_n\preceq y\preceq z_n,\,\, n=1,2,3,\cdots.
$$
Taking the limit as $n\rightarrow\infty$ in the above inequality with respect to the norm $\left\Vert \cdot\right\Vert _{C_{1-\,\gamma ;\Psi }\left(  J, \R\right)}$, we obtain
$$
w^*\preceq y\preceq z^*~\text{ in}~ C_{1-\,\gamma ;\Psi }\left(  J, \R\right).
$$
Therefore,  $w^*$ and $z^*$ are the  minimal solution and the maximal solution respectively in $ C_{1-\,\gamma ;\Psi }\left(  J, \R\right) $ of the nonlinear  $\Psi$-Hilfer  BVP \eqref{eq511}-\eqref{eq512}.
 \end{proof}
  
 \begin{theorem}\label{th52}
 Suppose that  the conditions of Theorem \ref{thm541} hold, and let there exists a constant 
$
 \tilde{L}\in\left[ 0,\, \frac{1}{(\Psi(T)-\Psi(0))^{\alpha}}\Omega  ^{-1}\right),
$
where
	\begin{align}\label{er9}
\Omega&=  B(\alpha-\gamma+1,\gamma)  \frac{r\,E_{\alpha,\,\gamma}\left( M\left(\Psi(T)-\Psi(0)\right)^{\alpha}\right) E_{\alpha,\,\alpha+1-\,\gamma}\left( M\left(\Psi(T)-\Psi(0)\right)^{\alpha}\right)}{\left[ 1-r\, E_{\alpha,\,1}\left( M\left(\Psi(T)-\Psi(0)\right)^{\alpha}\right)\right] } \nonumber \\
&\qquad\qquad + B(\alpha,\gamma)E_{\alpha,\,\alpha}\left( M\left(\Psi(T)-\Psi(0)\right)^{\alpha}\right).
	\end{align}
Further, assume that $f$ satisfies 
\begin{equation}\label{eq537}
f(t,\,x_2)-f(t,\,x_1)\leq\tilde{L}(x_2-x_1),~\text{for any } x_1, x_2\in\R\,\, \text{and}\,\, x_1\leq x_2.
\end{equation}
Then, the $\Psi$-Hilfer  BVP \eqref{eq511}-\eqref{eq512} has a unique solution $y^*$ in $[w_0,\,z_0]$. Moreover, for each $y_0\in[w_0,\,z_0]$ the iterative sequence defined by
\begin{align*}
 y_n(t)&=\frac{r\,(\Psi(t)-\Psi(0))^{\gamma-1}E_{\alpha,\,\gamma}\left(M (\Psi(t)-\Psi(0))^\alpha \right) }{\left[ 1-r\,E_{\alpha,\,1}\left(M (\Psi(T)-\Psi(0))^\alpha \right)\right]  }\\
 &\quad \times\int_{0}^{T} \Psi'(s)  (\Psi(T)-\Psi(s))^{\alpha-\gamma}  E_{\alpha,\,\alpha+1-\,\gamma}\left(M (\Psi(T)-\Psi(s))^\alpha \right) f(s,\,y_{n-1}(s)) ds\\
 &\qquad+\int_{0}^{t} \Psi'(s)  (\Psi(t)-\Psi(s))^{\alpha-1}  E_{\alpha,\,\alpha}\left(M (\Psi(t)-\Psi(s))^\alpha \right) f(s,\,y_{n-1}(s)) ds,~ n=1,2,3,\cdots,
\end{align*}
is such that
$$
\lim\limits_{n\rightarrow \infty}\left\| y_n-y^*\right\|_{C_{1-\,\gamma ;\Psi }\left( J,\,\R\right) }= 0
$$
and
$$
\left\| y_n-y^*\right\|_{C_{1-\,\gamma ;\Psi }\left( J,\,\R\right) }\leq  \varrho^n \left\| z_0-w_0\right\|_{C_{1-\,\gamma ;\Psi }\left( J,\,\R\right) },
$$
where
\begin{align}\label{rho}
\varrho=&\Omega\,(\Psi(T)-\Psi(0))^{\alpha}\,\tilde{L}.
\end{align}
  \end{theorem}
\begin{proof}
Let any $w_1, z_1\in [w_0, z_0]$ with $w_1\preceq z_1$. Then using the definition of operator $\mathcal{A}$ defined in  \eqref{ie1} and the condition \eqref{eq537}, we obtain
\begin{align*}
&\left( \left(\Psi(t)-\Psi(0)\right)^{1-\,\gamma}\left( \mathcal{A}z_1(t)-\mathcal{A}w_1(t)\right)\right)  \\
&=\left\lbrace \frac{r\,E_{\alpha,\,\gamma}\left(M (\Psi(t)-\Psi(0))^\alpha \right) }{\left[ 1-r\,E_{\alpha,\,1}\left(M (\Psi(T)-\Psi(0))^\alpha \right)\right]  }\right.\\
 &\quad\left. \times \int_{0}^{T} \Psi'(s)  (\Psi(T)-\Psi(s))^{\alpha-\gamma}  E_{\alpha,\,\alpha+1-\,\gamma}\left(M (\Psi(T)-\Psi(s))^\alpha \right) f(s,\,z_{1}(s)) ds\right.\\
 &\qquad\left.+(\Psi(t)-\Psi(0))^{1-\,\gamma}\int_{0}^{t} \Psi'(s)  (\Psi(t)-\Psi(s))^{\alpha-1}  E_{\alpha,\,\alpha}\left(M (\Psi(t)-\Psi(s))^\alpha \right) f(s,\,z_{1}(s)) ds\right\rbrace \\
 &\qquad-\left\lbrace \frac{r\,E_{\alpha,\,\gamma}\left(M (\Psi(t)-\Psi(0))^\alpha \right) }{\left[ 1-r\,E_{\alpha,\,1}\left(M (\Psi(T)-\Psi(0))^\alpha \right)\right]  }\right.\\
&\quad\left. \times \int_{0}^{T} \Psi'(s)  (\Psi(T)-\Psi(s))^{\alpha-\gamma}  E_{\alpha,\,\alpha+1-\,\gamma}\left(M (\Psi(T)-\Psi(s))^\alpha \right) f(s,\,w_{1}(s)) ds\right.\\
&\qquad\left.+(\Psi(t)-\Psi(0))^{1-\,\gamma}\int_{0}^{t} \Psi'(s)  (\Psi(t)-\Psi(s))^{\alpha-1}  E_{\alpha,\,\alpha}\left(M (\Psi(t)-\Psi(s))^\alpha \right) f(s,\,w_{1}(s)) ds\right\rbrace \\
&\leq\frac{\tilde{L}\, r\,E_{\alpha,\,\gamma}\left(M (\Psi(t)-\Psi(0))^\alpha \right) }{\left[ 1-r\,E_{\alpha,\,1}\left(M (\Psi(T)-\Psi(0))^\alpha \right)\right]  }\int_{0}^{T} \Psi'(s)  (\Psi(T)-\Psi(s))^{\alpha-\gamma} E_{\alpha,\,\alpha+1-\,\gamma}\left(M (\Psi(T)-\Psi(s))^\alpha \right) \\
&  \qquad \times (\Psi(s)-\Psi(0))^{\gamma-1} (\Psi(s)-\Psi(0))^{1-\,\gamma} \left[ z_{1}(s)-w_{1}(s)\right]  ds \\
& \qquad + \tilde{L}(\Psi(t)-\Psi(0))^{1-\,\gamma}\int_{0}^{t} \Psi'(s)  (\Psi(t)-\Psi(s))^{\alpha-1} E_{\alpha,\,\alpha}\left(M (\Psi(t)-\Psi(s))^\alpha \right)\\
 & \qquad \times  (\Psi(s)-\Psi(0))^{\gamma-1} (\Psi(s)-\Psi(0))^{1-\,\gamma}\left[ z_{1}(s)-w_{1}(s)\right]  ds  \\
&\leq\frac{\tilde{L}\, r\,E_{\alpha,\,\gamma}\left(M (\Psi(T)-\Psi(0))^\alpha \right) }{\left[ 1-r\,E_{\alpha,\,1}\left(M (\Psi(T)-\Psi(0))^\alpha \right)\right]  }E_{\alpha,\,\alpha+1-\,\gamma}\left(M (\Psi(T)-\Psi(0))^\alpha \right)\left\Vert z_1-w_1\right\Vert _{C_{1-\,\gamma ;\Psi }\left( J,\,\R\right) }\\
&\qquad \times \Gamma(\alpha-\gamma+1)\left[\mathcal{I}_{0+}^{\alpha-\,\gamma+1;\,\Psi }   (\Psi(t)-\Psi(0))^{\gamma-1}  \right] _{t=T} \\
&+\tilde{L}\, (\Psi(T)-\Psi(0))^{1-\,\gamma}E_{\alpha,\,\alpha}\left(M (\Psi(T)-\Psi(0))^\alpha \right)\left\Vert z_1-w_1\right\Vert _{C_{1-\,\gamma ;\Psi }\left( J,\,\R\right)}\Gamma(\alpha)\,\mathcal{I}_{0+}^{\alpha;\,\Psi }   (\Psi(t)-\Psi(0))^{\gamma-1}\\
&\leq\frac{\tilde{L}\, r\,E_{\alpha,\,\gamma}\left(M (\Psi(T)-\Psi(0))^\alpha \right) }{\left[ 1-r\,E_{\alpha,\,1}\left(M (\Psi(T)-\Psi(0))^\alpha \right)\right]  }E_{\alpha,\,\alpha+1-\,\gamma}\left(M (\Psi(T)-\Psi(0))^\alpha \right)\left\Vert z_1-w_1\right\Vert _{C_{1-\,\gamma ;\Psi }\left( J,\,\R\right) }\\
&\qquad \times \Gamma(\alpha-\gamma+1)\frac{\Gamma(\gamma)}{\Gamma(\alpha+1)}(\Psi(T)-\Psi(0))^\alpha+\tilde{L}\, (\Psi(T)-\Psi(0))^{1-\,\gamma} E_{\alpha,\,\alpha}\left(M (\Psi(T)-\Psi(0))^\alpha \right)\\
&\qquad \times \left\Vert z_1-w_1\right\Vert _{C_{1-\,\gamma ;\Psi }\left( J,\,\R\right)}\Gamma(\alpha)\frac{\Gamma(\gamma)}{\Gamma(\alpha+\gamma)}(\Psi(T)-\Psi(0))^{\alpha+\gamma-1}\\
&=\tilde{L}\,(\Psi(T)-\Psi(0))^\alpha\left\Vert z_1-w_1\right\Vert _{C_{1-\,\gamma ;\Psi }\left( J,\,\R\right)}\\
&\qquad \times \left\lbrace  B(\alpha-\gamma+1,\gamma) \frac{ r\,E_{\alpha,\,\gamma}\left(M (\Psi(T)-\Psi(0))^\alpha \right) }{1-r\,E_{\alpha,\,1}\left(M (\Psi(T)-\Psi(0))^\alpha \right) }E_{\alpha,\,\alpha+1-\,\gamma}\left(M (\Psi(T)-\Psi(0))^\alpha \right)\right.\\
&\qquad\left.+B(\alpha,\gamma)E_{\alpha,\,\alpha}\left(M (\Psi(T)-\Psi(0))^\alpha \right)\right\rbrace .
\end{align*}
Therefore,
\begin{align*}
\left\Vert \mathcal{A}z_1-\mathcal{A}w_1\right\Vert _{C_{1-\,\gamma ;\Psi }\left( J,\,\R\right)}
&\leq \tilde{L}\,(\Psi(T)-\Psi(0))^\alpha\left\Vert z_1-w_1\right\Vert _{C_{1-\,\gamma ;\Psi }\left( J,\,\R\right)}\Omega.
\end{align*}
Using \eqref{rho}, above inequality reduces to
$$
\left\Vert \mathcal{A}z_1-\mathcal{A}w_1\right\Vert _{C_{1-\,\gamma ;\Psi }\left( J,\,\R\right)}\leq\varrho\left\Vert z_1-w_1\right\Vert_{C_{1-\,\gamma ;\Psi }\left( J,\,\R\right)},~w_0\preceq w_1\preceq z_1\preceq z_0.
$$
Consider the sequences  $\{z_n\}$ and $\{w_n\}$ defined in Theorem \ref{thm541}. Then $z_n=\mathcal{A}z_{n-1}$ and $w_n=\mathcal{A}w_{n-1} (n=1, 2, 3,\cdots)$. By repeated application of the above inequality, we obtain
\begin{align}\label{er6}
\left\Vert z_n -w_n\right\Vert _{C_{1-\,\gamma ;\Psi }\left( J,\,\R\right)}
&=\left\Vert \mathcal{A}z_{n-1}-\mathcal{A}w_{n-1}\right\Vert _{C_{1-\,\gamma ;\Psi }\left( J,\,\R\right)}\nonumber\\
&\leq\varrho\left\Vert z_{n-1}-w_{n-1}\right\Vert _{C_{1-\,\gamma ;\Psi }\left( J,\,\R\right)}=\varrho\left\Vert \mathcal{A}z_{n-2}-\mathcal{A}w_{n-2}\right\Vert _{C_{1-\,\gamma ;\Psi }\left( J,\,\R\right)}\nonumber\\
&\leq \varrho^2\left\Vert z_{n-2}-w_{n-2}\right\Vert _{C_{1-\,\gamma ;\Psi }\left( J,\,\R\right)}=\varrho^2\left\Vert \mathcal{A}z_{n-3}-\mathcal{A}w_{n-3}\right\Vert _{C_{1-\,\gamma ;\Psi }\left( J,\,\R\right)}\nonumber\\
&\leq\cdots\leq \varrho^n\left\Vert z_{0}-w_{0}\right\Vert _{C_{1-\,\gamma ;\Psi }\left( J,\,\R\right)}. 
\end{align}
Using the condition on  $\tilde{L}$, we obtain $0\leq\varrho<1$. This implies that $\varrho^n\rightarrow0$ as $n\rightarrow\infty$. Therefore, from the inequality \eqref{er6}, it follows  that 
$$\lim\limits_{n\rightarrow\infty}\left\Vert z_n -w_n\right\Vert _{C_{1-\,\gamma ;\Psi }\left( J,\,\R\right)}\rightarrow0.
$$ 
By applying Theorem \ref{thm541}, there exists minimal solution $w^*$ and maximal solution $z^*$ in $[w_0, z_0]$
such that 
\begin{equation}\label{er1}
\mathcal{A}w^*=w^* ~~\text{and}~~\mathcal{A}z^*=z^*.
\end{equation}
Further, 
\begin{equation}\label{er2}
\lim\limits_{n\rightarrow\infty}\left\Vert w_n -w^*\right\Vert _{C_{1-\,\gamma ;\Psi }\left( J,\,\R\right)}=0
\,~\text{and}~
\lim\limits_{n\rightarrow\infty}\left\Vert z_n -z^*\right\Vert _{C_{1-\,\gamma ;\Psi }\left( J,\,\R\right)}=0.
\end{equation}
Using the equations in \eqref{er2} and the continuity of norm, we have
$$
0=\lim\limits_{n \rightarrow \infty}\left\Vert z_n -w_n\right\Vert _{C_{1-\,\gamma ;\Psi }\left( J,\,\R\right)}=\left\Vert z^* -w^*\right\Vert _{C_{1-\,\gamma ;\Psi }\left( J,\,\R\right)}.
$$
This gives 
\begin{equation}\label{er3}
z^*=w^*:=y^*~\text{ in}~C_{1-\,\gamma ;\, \Psi }\left( J,\,\R\right).
\end{equation}
From equations \eqref{er1} and \eqref{er3}, we have
$$
\mathcal{A}y^*=y^*.
$$
Thus, we have  a unique $y^*\in\left[ w_{0},\,z_{0}\right] $ such that 
\begin{equation}\label{er7}
\lim\limits_{n\rightarrow\infty}\left\Vert w_n -y^*\right\Vert _{C_{1-\,\gamma ;\Psi }\left( J,\,\R\right)}=0
\,~\text{and}~
\lim\limits_{n\rightarrow\infty}\left\Vert z_n -y^*\right\Vert _{C_{1-\,\gamma ;\Psi }\left( J,\,\R\right)}=0.
\end{equation}
 Since $\{w_n\}_{n=1}^\infty\subseteq C_{1-\,\gamma ;\, \Psi }\left( J,\,\R\right)$ is increasing bounded sequence  and $\{z_n\}_{n=1}^\infty\subseteq C_{1-\,\gamma ;\, \Psi }\left( J,\,\R\right)$ is decreasing bounded sequence, from \eqref{er7} it follows that
\begin{equation}\label{er4}
w_n\preceq y^*\preceq z_n~\text{ in}~ C_{1-\,\gamma ;\Psi }\left(  J, \R\right).
\end{equation} 
 For each $y_0\in \left[ w_{0},\,z_{0}\right]$, consider the iterative sequence $y_n=\mathcal{A}y_{n-1}$. Then using the increasing nature of an operator $\mathcal{A}$ and the definitions of $w_n$ and $z_n$, we obtain
 \begin{equation}\label{er5}
 w_n\preceq y_n\preceq z_n~\text{ in} ~ C_{1-\,\gamma ;\Psi }\left(  J, \R\right).
\end{equation} 
 Using the inequalities \eqref{er6}, \eqref{er4} and  \eqref{er5}, for each $t\in J$, we have
\begin{align*}
\left( \left(\Psi(t)-\Psi(0)\right)^{1-\,\gamma}\left( y_n(t) -y^*(t)\right)\right) &\leq \left( \left(\Psi(t)-\Psi(0)\right)^{1-\,\gamma}\left( z_n(t) -w_n(t)\right)\right)\\
&\leq \left| \left( \left(\Psi(t)-\Psi(0)\right)^{1-\,\gamma}\left( z_n(t) -w_n(t)\right)\right)\right| \\
&\leq\left\Vert z_n -w_n\right\Vert _{C_{1-\,\gamma ;\Psi }\left( J,\,\R\right)}\\
&\leq \varrho^n\left\Vert z_{0}-w_{0}\right\Vert _{C_{1-\,\gamma ;\Psi }\left( J,\,\R\right)}.
\end{align*}
Therefore,
 \begin{equation}\label{er8}
\left\Vert y_n -y^*\right\Vert _{C_{1-\,\gamma ;\Psi }\left( J,\,\R\right)}\leq \varrho^n\left\Vert z_{0}-w_{0}\right\Vert _{C_{1-\,\gamma ;\Psi }\left( J,\,\R\right)}.
\end{equation} 
From the above inequality it follows that 
$$
\lim\limits_{n \rightarrow \infty}\left\Vert y_n -y^*\right\Vert _{C_{1-\,\gamma ;\Psi }\left( J,\,\R\right)}=0.
$$
Observe that the inequality \eqref{er8} gives the error bound with respect to the $\left\Vert \cdot\right\Vert _{C_{1-\,\gamma ;\Psi }\left( J,\,\R\right)}$ between approximation $y_n$ and the exact solution $y^*$ of the nonlinear $\Psi$-Hilfer  BVP \eqref{eq511}-\eqref{eq512}.
\end{proof}

\section{Example}
We consider a specific case of the problem \eqref{eq511}-\eqref{eq512} to illustrate the main results that we acquired. 
\begin{ex}
Consider the following   BVP for the nonlinear Caputo FDEs
\begin{align}
& ^C \mathcal{D}^{\frac{1}{2}}_{0 +}y(t)=\frac{\sqrt{\pi}}{10}-\frac{\sqrt{t}+1}{25}+\frac{1}{25}\sin\left( \frac{\sqrt{t}+1}{5}\right) +\frac{1}{25}\left( 5 y(t)-\sin\left(y(t)\right) \right) , ~t \in  (0,\,1],  ~\label{eq511a}\\
& y(0)= \frac{1}{2}\,y(1).\label{eq512b}
\end{align}
\end{ex}
One can verify that $y^*(t)=\frac{\sqrt{t}+1}{5},  ~t \in  [0,\,1]$  is an exact solution of the  problem \eqref{eq511a}-\eqref{eq512b}.
Comparing the above problem with the nonlinear $\Psi$-Hilfer  BVP  \eqref{eq511}-\eqref{eq512}, we obtain
\begin{equation}\label{6.3}
\alpha=\frac{1}{2}, \beta=1, \gamma=\alpha+\beta(1-\alpha)=1, M=0, r=\frac{1}{2}, T=1 ~\text{and}~ \Psi(t)=t, t\in[0,\,1].
\end{equation}
In this case the weighted space $C_{1-\,\gamma ;\Psi }\left( [0,\,1],\,\R\right)$ reduces to the space of continuous functions $C\left( [0,\,1],\,\R\right)$ endowed with the supremum norm.

Define the function $f: [0,\,1]\times\R\rightarrow\R$ by
\begin{equation}\label{6.4}
f(t,y)=\frac{\sqrt{\pi}}{10}-\frac{\sqrt{t}+1}{25}+\frac{1}{25}\sin\left( \frac{\sqrt{t}+1}{5}\right) +\frac{1}{25}\left( 5 y-\sin y \right).
\end{equation}
(1) For any $k\in\R$ consider the function $\tilde{f}:\R\rightarrow\R$ defined by 
$$
\tilde{f}(y)=k+\frac{1}{25}\left( 5 y-\sin y \right), \,y\in \R.
$$
Then, 
$
\tilde{f}'(y)=\frac{1}{25}\left( 5-\cos y \right)\geq 0,~ \text{for all }~y\in \R.
$ Therefore the function $\tilde{f}$ is increasing on $\R$ \,for any real $k$.  This implies the function $f: [0,\,1]\times\R\rightarrow\R$ defined in \eqref{6.4} is increasing in $y\in \R$ for each $t\in[0, 1]$.\\
(2) Let any $y_1, y_2\in\R$ with $y_1\leq y_2$. Then
\begin{align}\label{6.5}
f(t, y_2)-f(t, y_1)
&=\frac{1}{25}\left\lbrace 5(y_2-y_1)-\left( \sin y_2 - \sin y_1\right)\right\rbrace \nonumber \\
&\leq\frac{1}{25}\left\lbrace 5|y_2-y_1|+| \sin y_2 - \sin y_1|\right\rbrace.
\end{align}
Since $\sin y$ is continuous and differentiable on the interval $[y_1, y_2]$, applying  mean value theorem, there exists $\tilde{y}\in[y_1, y_2]$ such that
$$
\frac{\sin y_2-\sin y_1}{ y_2-y_1}=\cos \tilde{y}.
$$
This implies $| \sin y_2 - \sin y_1|\leq |y_2-y_1|$. Therefore, the inequality \eqref{6.5} reduces to 
\begin{align}\label{6.6}
f(t, y_2)-f(t, y_1)
&\leq\frac{1}{25}\left\lbrace 5|y_2-y_1|+|  y_2 -  y_1|\right\rbrace  \nonumber \\
&\leq \frac{6}{25}\left( y_2-y_1 \right), ~\text{for any} ~y_1, y_2\in\R ~\text{with} ~ y_1\leq y_2.
\end{align}
 Comparing the above inequality with \eqref{eq537}, we have $\tilde{L}=\frac{6}{25}$.\\
(3) Next, we prove that $
\tilde{L}\in\left[ 0,\, \Omega  ^{-1}\right)
$ where $\Omega$ is defied in  \eqref{er9}. Using the values given in   \eqref{6.3}, the equation \eqref{er9} reduces to
$$
\Omega=  B\left( \frac{1}{2}-1+1,1\right)   \frac{\frac{1}{2}\,E_{\frac{1}{2},\,1}\left( 0\right) E_{\frac{1}{2},\,\frac{1}{2}+1-\,1}\left( 0\right)}{\left[ 1-\frac{1}{2}\, E_{\frac{1}{2},\,1}\left( 0\right)\right] } + B\left( \frac{1}{2},1\right) E_{\frac{1}{2},\,\frac{1}{2}}\left( 0\right).
$$
Using Lemma \ref{lem28}, we obtain
\begin{equation}\label{6.5a}
E_{n_1,\,n_2}(0)=\frac{1}{\Gamma\left( n_2\right)}, ~n_1,~n_2>0.
\end{equation} 
Therefore
$$
\Omega=  \frac{\Gamma\left( \frac{1}{2}\right)\Gamma\left( 1\right)}{\Gamma\left( \frac{1}{2}+1\right)} \frac{\frac{1}{2}\,\frac{1}{\Gamma\left(1\right)}\frac{1}{\Gamma\left( \frac{1}{2}\right)}}{\left[ 1-\frac{1}{2}\, \frac{1}{\Gamma\left( 1\right)} \right]} +\frac{\Gamma\left( \frac{1}{2}\right)\Gamma\left( 1\right)}{\Gamma\left( \frac{1}{2}+1\right)}  \frac{1}{\Gamma\left( \frac{1}{2}\right)}=\frac{4}{\sqrt{\pi}}.
$$
Note that, for $\tilde{L}=\frac{6}{25} ~\mbox{and}~ \Omega=\frac{4}{\sqrt{\pi}}, $ we have $\tilde{L}\in\left[ 0,\, \Omega  ^{-1}\right)$.\\
(4) Define $z_0(t)=\sqrt{t}+1 ,\, t\in[0,1]$ and  $w_0(t)=\frac{-(\sqrt{t}+1)}{6},\, t\in[0,1]$. Then,
\begin{align*}
h_1(t):&=\, ^C \mathcal{D}^{\frac{1}{2}}_{0 +}z_0(t)=\frac{\sqrt{\pi}}{2},\\
h_2(t):&= f(t,z_0(t))\\
&=\frac{\sqrt{\pi}}{10}-\frac{\sqrt{t}+1}{25}+\frac{1}{25}\sin\left( \frac{\sqrt{t}+1}{5}\right) +\frac{1}{25}\left( 5 \left( \sqrt{t}+1\right) -\sin \left( \sqrt{t}+1\right)  \right),\\
h_3(t):&=\, ^C \mathcal{D}^{\frac{1}{2}}_{0 +}w_0(t)=\frac{-\sqrt{\pi}}{12},\\
h_4(t) :&= f(t, w_0(t))\\
&=\frac{\sqrt{\pi}}{10}-\frac{\sqrt{t}+1}{25}+\frac{1}{25}\sin\left( \frac{\sqrt{t}+1}{5}\right) +\frac{1}{25}\left( 5 \left( \frac{-(\sqrt{t}+1)}{6}\right) -\sin \left( \frac{-(\sqrt{t}+1)}{6}\right)  \right).
\end{align*}
\begin{figure}[!htb]
 	\begin{minipage}{0.5\textwidth}
 	\centering
 		\includegraphics[width=0.9\textwidth]{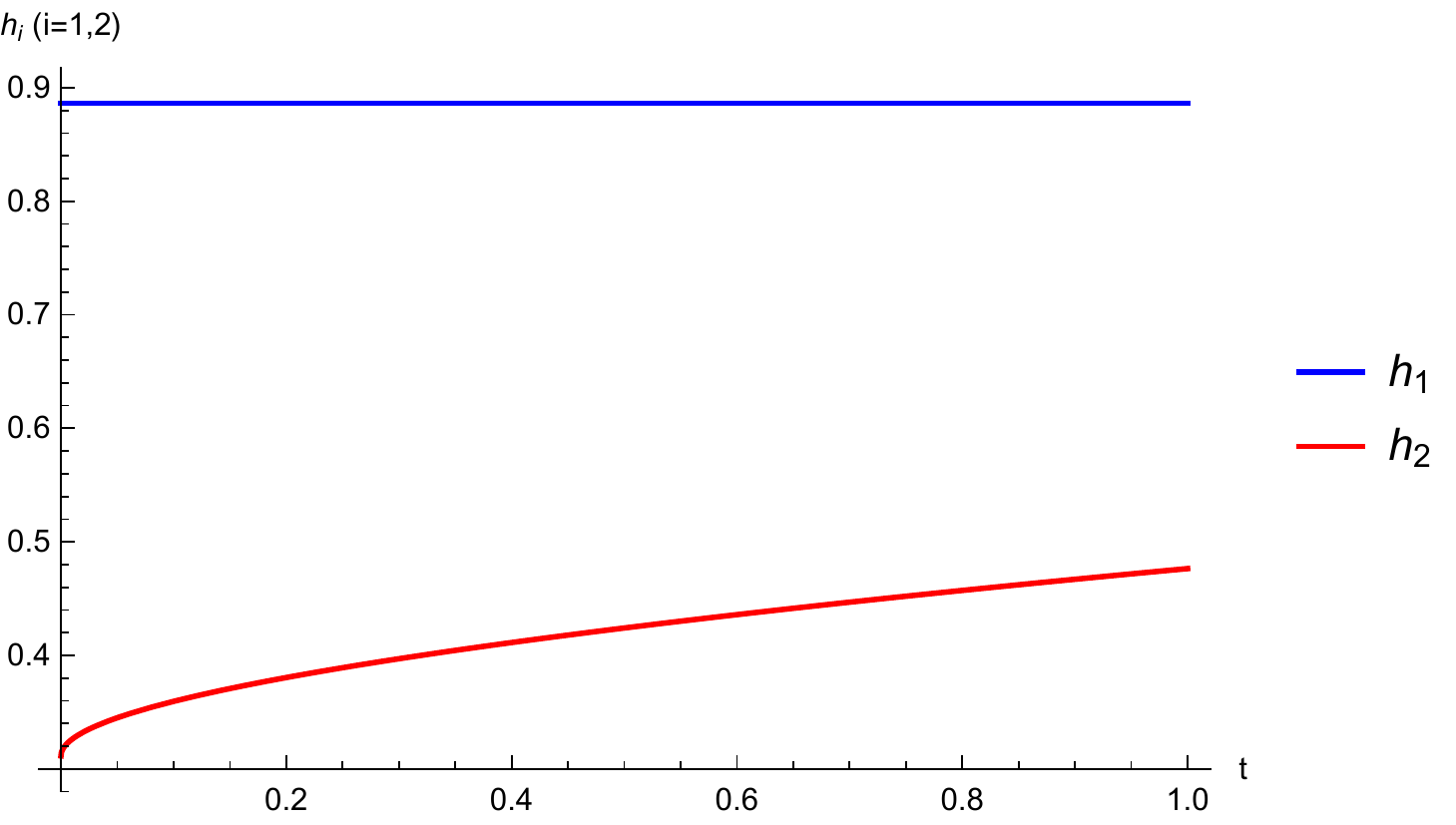}
 	\caption{$z_0$ is an upper solution}
 	\end{minipage}\hfill
	\begin{minipage}{0.5\textwidth}
			\centering
		\includegraphics[width=0.9\textwidth]{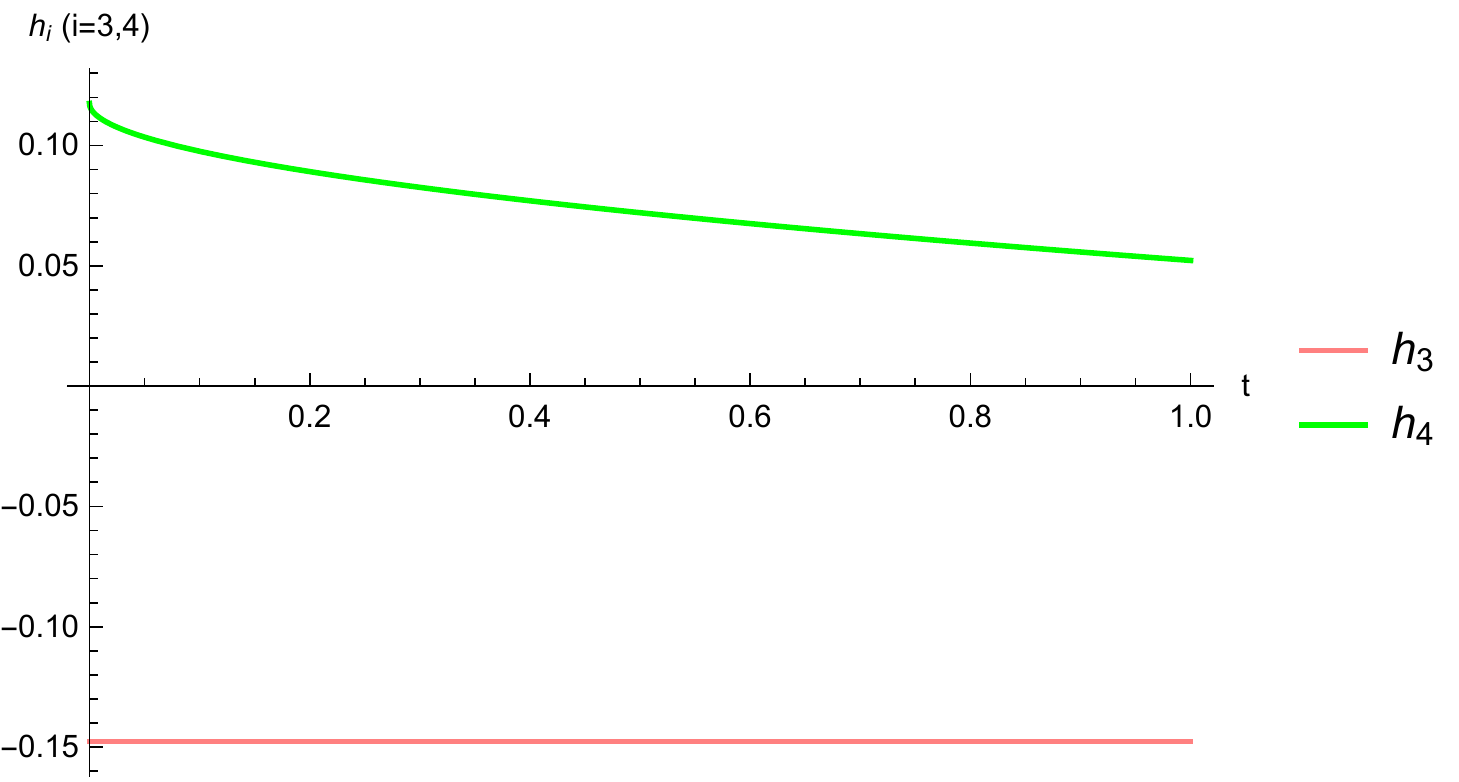}
		\caption{$w_0$ is a lower solution}
	\end{minipage}
\end{figure}
\newline
From Figure 1,  it follows that 
$$
^C \mathcal{D}^{\frac{1}{2}}_{0 +}z_0(t)\geq f(t, z_0(t)),\, t\in[0,1].
$$
Further, using the definition of $z_0$, we have $z_0(0)=\frac{1}{2}z_0(1)$  and hence $b_{z_0}(t)=0$. This proves that $z_0(t)=\sqrt{t}+1, \, t\in[0,1]$ is an upper solution  of the problem \eqref{eq511a}-\eqref{eq512b}. Again, from Figure 2,  it follows that

$$
^C \mathcal{D}^{\frac{1}{2}}_{0 +}w_0(t)\leq f(t, w_0(t)),\, t\in[0,1].
$$ 
Further, by definition of  $w_0$, we have
$w_0(0)=\frac{1}{2}w_0(1)$ and $a_{w_0}(t)=0$. This implies $w_0(t)=-\frac{\sqrt{t}+1}{6}, \, t\in[0,1]$ is a lower solution of the problem \eqref{eq511a}-\eqref{eq512b}.  Since all the assumptions of Theorem \ref{th52} are satisfied, it guarantee the existence of  a unique solution $y^*$ in $[w_0, z_0]$   of the problem \eqref{eq511a}-\eqref{eq512b}.  Indeed, 
$$
-\frac{\sqrt{t}+1}{6}\leq\frac{\sqrt{t}+1}{5}\leq\sqrt{t}+1, ~\text{for all }~ t\in [0,1],
$$
implies 
$$
w_0(t)\leq y^*(t)\leq z_0(t),~ t\in [0,1].
$$
This implies $y^*\in[w_0, z_0] $, where $y^*,~w_0$ and $z_0$ defined above, respectively are the exact, lower and upper solutions of the problem \eqref{eq511a}-\eqref{eq512b}.\\
(5) Using  \eqref{6.3} and \eqref{6.5a}, the sequence $\{y_n\}_{n=1}^\infty$ defined in the Theorem \ref{th52} reduces to 
$$
y_n(t)=\frac{1}{\sqrt{\pi}}\int_{0}^{1} \left( 1-s\right) ^{\frac{-1}{2}} f(s, y_{n-1}(s))ds+\frac{1}{\sqrt{\pi}}\int_{0}^{t} \left( t-s\right) ^{\frac{-1}{2}}f(s, y_{n-1}(s))ds
$$
where $y_0\in[w_0, z_0]$. By Theorem \ref{th52},
 \begin{equation}\label{6.10}
\left\Vert y_n -y^*\right\Vert _{C\left( [0,1],\,\R\right)}\leq \varrho^n\left\Vert z_{0}-w_{0}\right\Vert _{C\left( [0,1],\,\R\right)}.
\end{equation} 
Using \eqref{6.3} and the values of $\Omega$ and $\tilde{L}$ determined above, from \eqref{rho}, we obtain 
$$
\rho=\Omega\tilde{L}=\frac{4}{\sqrt{\pi}}\frac{6}{25}=\frac{24}{25\sqrt{\pi}}.
$$
Thus, from the inequality \eqref{6.10}, we have
\begin{align*}
\left\Vert y_n -y^*\right\Vert _{C\left( [0,1],\,\R\right)}
&\leq \left( \frac{24}{25\sqrt{\pi}}\right) ^n \underset{t\in[0,1]}{\sup}\left| z_0(t)-w_0(t)\right| \\
&= \left( \frac{24}{25\sqrt{\pi}}\right) ^n \underset{t\in[0,1]}{\sup}\left| \sqrt{t}+1-\left( -\frac{\sqrt{t}+1}{6}\right) \right| \\
&= \left( \frac{24}{25\sqrt{\pi}}\right) ^n \frac{7}{6}\,\underset{t\in[0,1]}{\sup}\left| \sqrt{t}+1\right|.
\end{align*}
Therefore,
\begin{equation}\label{6.11}
\left\Vert y_n -y^*\right\Vert _{C\left( [0,1],\,\R\right)}\leq \frac{7}{3} \left( \frac{24}{25\sqrt{\pi}}\right) ^n.
\end{equation}
The inequality \eqref{6.11} gives the error between $n^{th}$ approximation $y_n$  and exact solution $y^*$  of the problem \eqref{eq511a}-\eqref{eq512b}. Since $\frac{24}{25\sqrt{\pi}}<1$,  it follows that $y_n \rightarrow y^*$ in $C\left( [0,1],\,\R\right)$ as $n\rightarrow\infty$.


\section*{Acknowledgement}
The second author  acknowledges the Science and Engineering Research Board (SERB), New Delhi, India for the Research Grant (Ref: File no. EEQ/2018/000407).

\end{document}